\documentclass[12pt,english]{article}
\usepackage{mathptmx}

\usepackage[T1]{fontenc}
\usepackage[latin9]{inputenc}
\usepackage[a4paper]{geometry}
\usepackage{babel}
\usepackage{amsmath}
\usepackage{amsthm}
\usepackage{amssymb}
\usepackage{setspace}
\usepackage[numbers,square,sort,comma,numbers]{natbib}
\usepackage[unicode=true,
 bookmarks=true,bookmarksnumbered=false,bookmarksopen=false,
 breaklinks=false,pdfborder={0 0 1},backref=false,colorlinks=false]
 {hyperref}
\hypersetup{pdftitle={Use Lyx for typesetting: A template for empirical research projects},
 pdfauthor={First Author and Second Author}}

\makeatletter
\hypersetup{
    colorlinks=true,       
    linkcolor=blue,          
    citecolor=blue,        
    filecolor=blue,      
    urlcolor=blue,           
    breaklinks=true
}

\usepackage{booktabs}
\usepackage{longtable}
\usepackage{pdflscape}
\usepackage[square,sort,comma,numbers]{natbib}
\usepackage{amsthm}
\usepackage[title]{appendix}

\makeatother

\begin{document}
\title{\onehalfspacing{}Salter's question on the image of the Burau representation of $B_{4}$}
\author{Donsung Lee}
\date{December 19, 2024}

\maketitle
\medskip{}

\begin{abstract}
\begin{spacing}{0.9}
\noindent In 1974, Birman posed the problem of identifying the conditions
under which a matrix with Laurent polynomial entries lies in the image
of the Burau representation. Building on this, Salter, in 2021, refined
the inquiry to ask whether the central quotient of the Burau image
group coincides with the central quotient of a specific subgroup of
the unitary group. Assuming the faithfulness of the $B_{4}$ Burau,
we solve Salter's question negatively in the case $n=4$ constructing
counterexamples. Additionally, we offer two remarks on the faithfulness
of the $B_{4}$ Burau. First, we establish that the restriction to
the centralizer of a standard generator in $B_{4}$ is faithful modulo
$p$ for every prime $p$, extending both Smythe's result in 1979
and Moran's result in 1991. Second, we present a building-theoretic
criterion for the faithfulness.\vspace{1cm}

\medskip{}

\end{spacing}
\end{abstract}
\noindent $ $\theoremstyle{definition}
\theoremstyle{remark}
\newtheorem{theorem}{Theorem}[section]
\newtheorem{lemma}[theorem]{Lemma}
\newtheorem{corollary}[theorem]{Corollary}
\newtheorem{question}[theorem]{Question}
\newtheorem{remark}[theorem]{Remark}
\newtheorem{example}[theorem]{Example}
\newtheorem{definition}[theorem]{Definition}
\newtheorem{algorithm}[theorem]{Algorithm}
\newtheorem{hypothesis}[theorem]{Hypothesis}

\newtheorem*{ack}{Acknowledgements}
\newtheorem*{question1}{Question \textnormal{1}}
\newtheorem*{question2}{Question \textnormal{2}}
\newtheorem*{claim1}{Claim \textnormal{1}}
\newtheorem*{proofclaim1}{Proof of Claim \textnormal{1}}
\newtheorem*{claim2}{Claim \textnormal{2}}
\newtheorem*{proofclaim2}{Proof of Claim \textnormal{2}}
\newtheorem*{prooftheorem11}{Proof of Theorem \textnormal{1.1}}
\newtheorem*{prooftheorem12}{Proof of Theorem \textnormal{1.2}}
\newtheorem*{prooftheorem13}{Proof of Theorem \textnormal{1.3}}
\newtheorem*{prooftheorem14}{Proof of Theorem \textnormal{1.4}}
\newtheorem*{prooftheorem15}{Proof of Theorem \textnormal{1.5}}

\section{Introduction}

The \emph{Burau representation} is a fundamental linear representation
for \emph{braid groups} $B_{n}$, which has the generators $\sigma_{1},\,\sigma_{2},\cdots,\,\sigma_{n-1}$
\citep{MR2435235}. The representation has two formulations, the \emph{unreduced}
and \emph{reduced} Burau representations. For each integer $i$ such
that $1\le i\le n-1$, the unreduced representation $\beta_{n,\,u}$
is given explicitly by
\begin{align*}
\sigma_{i} & \mapsto\left(\begin{array}{cccc}
I_{i-1} & 0 & 0 & 0\\
0 & 1-t & t & 0\\
0 & 1 & 0 & 0\\
0 & 0 & 0 & I_{n-i-1}
\end{array}\right)\in\mathrm{GL}\left(n,\,\mathbb{Z}\left[t,\,t^{-1}\right]\right),
\end{align*}
while the reduced representation $\beta_{n,\,r}$ is given to $\mathrm{GL}\left(n-1,\,\mathbb{Z}\left[t,\,t^{-1}\right]\right)$
by
\begin{align*}
 & \sigma_{1}\mapsto\left(\begin{array}{ccc}
-t & 1\\
0 & 1\\
 &  & I_{n-3}
\end{array}\right),\:\sigma_{n-1}\mapsto\left(\begin{array}{ccc}
I_{n-3}\\
 & 1 & 0\\
 & t & -t
\end{array}\right),\\
 & \sigma_{i}\mapsto\left(\begin{array}{ccccc}
I_{i-2}\\
 & 1 & 0 & 0\\
 & t & -t & 1\\
 & 0 & 0 & 1\\
 &  &  &  & I_{n-i-2}
\end{array}\right),\;2\le i\le n-2.
\end{align*}

In 1974, Birman (\citep{MR0375281}, Research Problem 14) asked how
the image of $\beta_{n,\,u}$ is characterized in $\mathrm{GL}\left(n,\,\mathbb{Z}\left[t,t^{-1}\right]\right)$
as a group of matrices. In 1984, Squier in \citep{MR0727232} observed
that the elements in the image satisfy the \emph{unitarity}. If we
put
\begin{align*}
J_{n} & :=\left(\begin{array}{ccccc}
1 & -t^{-1} & -t^{-1} & \cdots & -t^{-1}\\
-t & 1 & -t^{-1} & \cdots & -t^{-1}\\
-t & -t & 1 & \cdots & -t^{-1}\\
\vdots & \vdots & \vdots & \ddots & \vdots\\
-t & -t & -t & \cdots & 1
\end{array}\right),
\end{align*}
then any element $A\in\beta_{n,\,u}\left(B_{n,\,u}\right)\subset\mathrm{GL}\left(n,\,\mathbb{Z}\left[t,\,t^{-1}\right]\right)$
satisfies the \emph{unitarity relation}
\begin{align}
\overline{A}J_{n}A^{T} & =J_{n},
\end{align}
where for a matrix $A$ with Laurent polynomial entries, $A\mapsto\overline{A}$
maps $t\mapsto t^{-1}$ for every entry. Define further for any Laurent
polynomial $f\left(t\right)\in\mathbb{Z}\left[t,\,t^{-1}\right]$,
$\overline{f\left(t\right)}:=f\left(t^{-1}\right)$. Since $\beta_{n,\,r}$
is the composition factor of $\beta_{n,\,u}$, we also have a unitarity
relation on $\beta_{n,\,r}$.

In 2021, Salter in \citep{MR4228497} pointed out several algebraic
properties satisfied by the elements in the Burau images and defined
the groups $\Gamma_{n}$ (resp. $\Gamma_{n}'$) as subgroups of the
unitary group associated with $\beta_{n,\,u}$ (resp. $\beta_{n,\,r}$)
based on these properties. For each $n$, denote by $\overline{\Gamma_{n}}$
(resp. $\overline{\Gamma_{n}'}$) the quotient of $\Gamma_{n}$ (resp.
$\Gamma_{n}'$) by its center via the quotient map $q_{n}:\Gamma_{n}\to\overline{\Gamma_{n}}$
(resp. $q_{n}':\Gamma_{n}'\to\overline{\Gamma_{n}'}$). Salter then
posed the question of whether the composition $q_{n}\circ\beta_{n,\,u}:B_{n}\to\overline{\Gamma_{n}}$
(resp. $q_{n}'\circ\beta_{n,\,r}:B_{n}\to\overline{\Gamma_{n}'}$)
is surjective for $n\ge2$. Denote $q_{n}\circ\beta_{n,\,u}$ (resp.
$q_{n}'\circ\beta_{n,\,r}$) by $\overline{\beta_{n,\,u}}$ (resp.
$\overline{\beta_{n,\,r}}$). Since the case $n=2$ is straightforward
to verify, we focus on the cases $n\ge3$.

Recently, the author in \citep{lee2024salters} demonstrated that
$\overline{\beta_{3,\,u}}$ is not surjective by constructing counterexamples
by using a tree-based algorithm. For the case $n=3$, these counterexamples
also establish that $\overline{\beta_{3,\,r}}$ is not surjective.

In this paper, assuming that $\beta_{4,\,u}$ (resp. $\beta_{4,\,r}$)
is faithful, we prove that $\overline{\beta_{4,\,u}}$ (resp. $\overline{\beta_{4,\,r}}$)
is not surjective. This is achieved by constructing two types of counterexamples:
the first (${\mathrm{Theorem\,\,1.1}}$) is hereditary from those
found in \citep{lee2024salters}, while the second (Theorem 1.2) is
newly discovered.

\noindent \begin{theorem}

Suppose that $\beta_{4,\,u}$ (resp. $\beta_{4,\,r}$) is faithful.
Then, there is an element $A\in\overline{\Gamma_{4}}$ (resp. $A\in\overline{\Gamma_{4}'}$)
that commutes with $\overline{\beta_{4,\,u}}\left(\left(\sigma_{3}\sigma_{2}\sigma_{3}\right)^{2}\right)$
(resp. $\overline{\beta_{4,\,r}}\left(\left(\sigma_{3}\sigma_{2}\sigma_{3}\right)^{2}\right)$)
and lies outside the image $\overline{\beta_{4,\,u}}\left(B_{4}\right)$
(resp. $\overline{\beta_{4,\,r}}\left(B_{4}\right)$).

\noindent \end{theorem}
\begin{theorem}

Suppose that $\beta_{4,\,r}$ is faithful. Then, there is an element
$A\in\overline{\Gamma_{4}'}$ that commutes with $\overline{\beta_{4,\,r}}\left(\sigma_{3}\right)$
and lies outside the image of $\overline{\beta_{4,\,r}}\left(B_{4}\right)$.

\noindent \end{theorem}

The proofs of these two theorems proceed by introducing hypotheses
about the structure of the image and showing that the faithfulness
of the $B_{4}$ Burau implies these hypotheses. Specifically, the
proof of Theorem 1.1 relies on facts about the structure of the braid
Torelli group of $B_{4}$ (see \citep{MR3323579,MR545151}), while
the proof of Theorem 1.2 involves an analysis of the centralizer of
$\sigma_{3}$ in $B_{4}$, denoted $C_{B_{4}}\left(\sigma_{3}\right)$.
Finding a generating set for centralizers in braid groups is an algorithmic
problem (see \citep{MR747249,MR2023190}). For $B_{4}$, we have
\begin{align}
C_{B_{4}}\left(\sigma_{3}\right) & =\left\langle \sigma_{1},\,\sigma_{3},\,\left(\sigma_{3}\sigma_{2}\sigma_{3}\right)^{2}\right\rangle .
\end{align}

Next, we turn our attention to the faithfulness of the $B_{4}$ Burau,
which serves as the premise for the two theorems. Establishing this
remains a well-known open problem with a long history. Define two
braids in $B_{4}$ by
\begin{align*}
b_{1} & :=\sigma_{1}\sigma_{3}^{-1},\;b_{2}:=\sigma_{1}\sigma_{2}\sigma_{3}\sigma_{1}^{-1}\sigma_{2}^{-1}\sigma_{1}^{-1}.
\end{align*}
and define two matrices in $\mathrm{GL}\left(3,\,\mathbb{Z}\left[t,\,t^{-1}\right]\right)$
by
\begin{align*}
S & :=\beta_{4,\,r}\left(b_{1}\right),\;T:=\beta_{4,\,r}\left(b_{2}\right).
\end{align*}

Birman \citep{MR0375281} showed that $\beta_{4,\,r}$ is faithful
if and only if the two matrices $S$ and $T$ generate a free group
of rank 2. Partial results regarding this freeness include the following:
$\left\langle S^{m},\,T^{n}\right\rangle $ is free of rank 2 when
$m,\,n\ge3$ (\citep{MR596856,MR3415244,MR3786776}); $\left\langle S,\,TST^{-1}\right\rangle $
is free of rank 2 (\citep{MR1113420}); $\left\langle S^{3},\,T^{3},\,\left(ST\right)^{3},\,\left(ST^{-1}\right)^{3}\right\rangle $
is free of rank 4 (\citep{MR4229222}). On the other hand, the mod
$p$ representation derived from $\beta_{4,\,r}$ has been shown to
be non-faithful for $p=2,\,3,\,5$ (\citep{MR1431138,MR1668343,gibson20234strandburauunfaithfulmodulo}).
In the process of proving Theorem 1.2, we establish the following
result, extending the work of \citep{MR1113420}. This is achieved
using the structure theory of $2\times2$ similitude groups over the
Laurent polynomial ring over general fields, as developed in \citep{lee2024buraurepresentationb3modulo}.

\noindent \begin{theorem}

The representation $\beta_{4,\,r}|_{C_{B_{4}}\left(\sigma_{3}\right)}$
is faithful. Moreover, this faithfulness is preserved modulo $p$
for every prime $p$. In particular, the matrices $S$ and $TST^{-1}$
generate a free group of rank 2 modulo $p$ for every $p$.

\noindent \end{theorem}

Theorem 1.3 has a corollary stating that the image of $\left\langle \left(\sigma_{1}\sigma_{2}\sigma_{1}\right)^{2},\,\left(\sigma_{3}\sigma_{2}\sigma_{3}\right)^{2}\right\rangle $
is free under the Burau representation modulo $p$ for every $p$
(Remark 3.8), extending a result by Smythe \citep{MR545151}.

Several authors \citep{MR1431138,MR1668343,MR3415244} have studied
the two-dimensional Bruhat-Tits building on which the Burau image
acts, relating the faithfulness problem to the geometry of buildings.
This connection will be explored in the final part of the paper. We
examine a large Bruhat-Tits building $X_{\mathbb{C}\left(t\right)}$
on which the unitary group, defined with $\mathbb{C}\left[t,\,t^{-1}\right]$
entries, acts. We show that the unipotent kernel $\mathcal{U}_{\mathbb{C}}$
(see Definition 4.4) acts simply transitively on the vertices of this
building. Under these conditions, the group $\mathcal{U}_{\mathbb{C}}$
admits a triangle presentation as a locally infinite $\widetilde{A_{2}}$
group (see \citep{MR1232965,MR1232966}). More generally, we establish
the following result.

\noindent \begin{theorem}

For any infnite cardinality $\kappa$, there is an $\widetilde{A_{2}}$
group where the number of vertices in the link of a vertex is $\kappa$.

\noindent \end{theorem}

Finally, we examine a locally finite connected simplicial subcomplex
$\Delta_{\mathbb{C}}\left(\mathcal{H}\right)$ of the building, which
is spanned by a finitely generated subgroup $\mathcal{H}$ of $\mathcal{U}_{\mathbb{C}}$
which contains a finite index subgroup of $\left\langle S,\,T\right\rangle $
up to some conjugation (see Section 4 for definitions). An analysis
of this subcomplex leads to the following criterion for faithfulness.
It is worth mentioning that $\Delta_{\mathbb{C}}\left(\mathcal{H}\right)$
is contained in a much smaller building $X_{\mathbb{Q}\left(i\right)\left(t\right)}$.

\noindent \begin{theorem}

If the subcomplex $\Delta_{\mathbb{C}}\left(\mathcal{H}\right)$ is
simply connected, the representation $\beta_{4,\,r}$ is faithful.

\noindent \end{theorem}

The rest of this paper is organized as follows. In Section 2, we prove
Theorem 1.1. In Section 3, we examine the stabilizer $C_{B_{4}}\left(\sigma_{3}\right)$,
and prove Theorem 1.3 and Theorem 1.2. In Section 4, we prove Theorem
1.4 and Theorem 1.5.

\section{Hereditary Counterexamples}

Define two vectors
\begin{align*}
v_{n} & :=\left(t,\,t^{2},\,\cdots,\,t^{n}\right),\;\overrightarrow{1_{n}}:=\left(1,\,1,\,\cdots,\,1\right)^{T}.
\end{align*}

Every matrix $A\in\beta_{n,\,u}\left(B_{n}\right)$ satisfies the
following.
\begin{align*}
 & A\overrightarrow{1_{n}}=\overrightarrow{1_{n}},\\
 & v_{n}A=v_{n},
\end{align*}
where we followed Salter's notation, except that the matrix transpose
was taken. This convention is due to \citep{MR2435235}. Salter also
noticed that the image of the evaluation of the Burau image at $t=1$
is isomorphic to the symmetric group $S_{n}$, by the permutation
representation. Then, he formally defined the target groups $\Gamma_{n},\,\Gamma_{n}'$
in \citep[Definition 2.5]{MR4228497}.

\noindent \begin{definition}

The group $\Gamma_{n}$ is defined as follows:
\begin{align*}
\Gamma_{n} & :=\left\{ A\in\mathrm{GL}\left(n,\,\mathbb{Z}\left[t,t^{-1}\right]\right)\::\:v_{n}A=v_{n},\:A\overrightarrow{1_{n}}=\overrightarrow{1_{n}},\:\overline{A}J_{n}A^{T}=J_{n},\:A|_{t=1}\in S_{n}\right\} .
\end{align*}

Define the group $\Gamma_{n}'$ to be the image of $\Gamma_{n}'$
when restricted to the $J_{n}$-orthogonal complement. Denote by $\overline{\Gamma_{n}}$
(resp. $\overline{\Gamma_{n}'}$) the quotient of $\Gamma_{n}$ (resp.
$\Gamma_{n}'$) by its center.

\noindent \end{definition}

In some contexts, we will also call $\overline{\Gamma_{n}}$ or $\overline{\Gamma_{n}'}$
the target group, if there is no possibility for confusion. Define
the map $\overline{\beta_{n,\,u}}:B_{n}\to\overline{\Gamma_{n}}$
(resp. $\overline{\beta_{n,\,r}}:B_{n}\to\overline{\Gamma_{n}'}$)
to be the composition of the quotient $q_{n}:\Gamma_{n}\to\overline{\Gamma_{n}}$
(resp. $q_{n}':\Gamma_{n}'\to\overline{\Gamma_{n}'}$) and the Burau
representation $\beta_{n,\,u}$ (resp. $\beta_{n,\,r}$). Here is
Salter's question \citep[Question 1.1]{MR4228497}.

\noindent \begin{question}

Are the maps $\overline{\beta_{n,\,u}}:B_{n}\to\overline{\Gamma_{n}}$
and $\overline{\beta_{n,\,r}}:B_{n}\to\overline{\Gamma_{n}'}$ surjective?

\noindent \end{question}

The author algorithmically constructed counterexamples for the case
$n=3$ \citep[Section 5]{lee2024salters}. By definition, the group
$\Gamma_{n+1}$ contains an isomorphic copy of $\Gamma_{n}$ for each
$n>1$, via the inclusion:
\begin{align*}
\Gamma_{n} & \hookrightarrow\left(\begin{array}{cc}
1\\
 & \Gamma_{n}
\end{array}\right)\subset\Gamma_{n+1},
\end{align*}
where the first row and column have 0 as the $\left(1,\,i\right)$-entry
and $\left(j,\,1\right)$-entry for every $i,\,j>1$. We refer to
this situation as the matrix in $\Gamma_{n+1}$ having the \emph{trivial
first row and column}. Similarly, the group $\Gamma_{n+1}'$ contains
an isomorphic copy of $\Gamma_{n}'$ for each $n>1$, via the inclusion:
\begin{align*}
\Gamma_{n}' & \hookrightarrow\left(\begin{array}{cc}
1\\
 & \Gamma_{n}'
\end{array}\right)\subset\Gamma_{n+1}',
\end{align*}
where the first row has 0 as $\left(1,\,i\right)$-entry for every
$i>1$ (we refer to this situation as the matrix in $\Gamma_{n+1}'$
having the \emph{trivial first row}) and the first column entries
$\left(j,\,1\right)$, $j>1$, are completely determined by the other
entries. Based on these observations, we formulate \emph{separation
hypotheses}.

\noindent \begin{hypothesis}

$ $
\begin{itemize}
\item $\left(\mathrm{SH}_{n}\right)$ If a matrix $A\in\beta_{n,\,u}\left(B_{n}\right)$
has the trivial first row and column, then $A$ is included in the
subgroup $\left\langle \beta_{n,\,u}\left(\sigma_{2}\right),\,\beta_{n,\,u}\left(\sigma_{3}\right),\,\cdots,\,\beta_{n,\,u}\left(\sigma_{n-1}\right)\right\rangle $.
\item $\left(\mathrm{SH}_{n}'\right)$ If a matrix $B\in\beta_{n,\,r}\left(B_{n}\right)$
has the trivial first row, then $B$ is included in the subgroup $\left\langle \beta_{n,\,r}\left(\sigma_{2}\right),\,\beta_{n,\,r}\left(\sigma_{3}\right),\,\cdots,\,\beta_{n,\,r}\left(\sigma_{n-1}\right)\right\rangle $.
\end{itemize}
\noindent \end{hypothesis}

Historically, Dehornoy in \citep{MR1381685} was the first to propose
hypotheses of this type. According to his notation, the faithfulness
of $\beta_{n,\,u}$ and $\left(\mathrm{SH}_{n}\right)$ both hold
if and only if $\mathcal{P}_{k}^{+}\left(B_{n}\right)$ holds for
every $k$. The hypothesis $\left(\mathrm{SH}_{3}\right)$ was shown
in \citep[Proposition 1, p. 22]{MR1381685}. The hypothesis $\left(\mathrm{SH}_{3}'\right)$
is also straightforward from the unitarity.

\noindent \begin{lemma}

For an integer $n\ge4$, assume the hypotheses $\left(\mathrm{SH}_{4}\right),\,\left(\mathrm{SH}_{5}\right),\,\cdots,\,\left(\mathrm{SH}_{n}\right)$
(resp. $\left(\mathrm{SH}_{4}'\right),\,\left(\mathrm{SH}_{5}'\right),\,\cdots,\,\left(\mathrm{SH}_{n}'\right)$).
Then, there is an element $A\in\overline{\Gamma_{n}}$ (resp. $A\in\overline{\Gamma_{n}'}$)
that commutes with $\beta_{n,\,u}\left(\left(\sigma_{n-1}\sigma_{n-2}\sigma_{n-1}\right)^{2}\right)$
(resp. $\beta_{n,\,r}\left(\left(\sigma_{n-1}\sigma_{n-2}\sigma_{n-1}\right)^{2}\right)$)
and lies outside the image $\overline{\beta_{n,\,u}}\left(B_{n}\right)$
(resp. $\overline{\beta_{n,\,r}}\left(B_{n}\right)$).

\noindent \end{lemma}
\begin{proof}For the case $n=4$, choose a counterexample $A\in\overline{\Gamma_{3}}$
which lies outside the image $\overline{\beta_{3,\,u}}\left(B_{3}\right)$
and choose a lift $\widehat{A}\in\Gamma_{3}$ for $A$. The inclusion
yields
\begin{align*}
 & \left[\begin{array}{cc}
1\\
 & \widehat{A}
\end{array}\right]\subset\overline{\Gamma_{n+1}},
\end{align*}
where the square bracket $\left[\right]$ means the matrix is included
in the projective linear group. By $\left(\mathrm{SH}_{4}\right)$,
this element must lie outside the image of $\overline{\beta_{4,\,u}}\left(B_{4}\right)$.
The higher $n$ cases are established by induction, and the cases
for the reduced Burau are handled in the same manner. $\qedhere$

\noindent \end{proof}

At first glance, the separation hypotheses might appear obvious. However,
they are far from trivial. Dehornoy proved the following fact.

\noindent \begin{lemma}

The faithfulness of the representation $\beta_{n+1,\,u}$ implies
$\left(\mathrm{SH}_{n}\right)$.

\noindent \end{lemma}
\begin{proof}See \citep[Lemma 3, p. 5]{MR1381685}. $\qedhere$

\noindent \end{proof}

Since the Burau representation is known to be non-faithful for $n>4$
\citep{MR1217079,MR1725480}, $\left(\mathrm{SH}_{n}\right)$ cannot
be proven through the faithfulness of $\beta_{n+1,\,u}$. However,
Lemma 2.5 implies that proving the negation of $\left(\mathrm{SH}_{n}\right)$
offers a new method to prove the non-faithfulness of $\beta_{n+1,\,u}$.
The validity of these hypotheses appears to be closely related to
the properties of the braid Torelli groups, which are the kernels
of the integral Burau representations from $B_{n}$ to $\mathrm{SL}\left(n,\,\mathbb{Z}\right)$,
or $\beta_{n,\,u}$ evaluated at $t=-1$ (see \citep{MR3323579}).

\noindent \begin{prooftheorem11}

\noindent By Lemma 2.4, it suffices to show that the faithfulness
of $\beta_{4,\,u}$ implies $\left(\mathrm{SH}_{4}\right)$. Suppose
that $\beta_{4,\,u}$ is faithful. Suppose a matrix $A\in\beta_{4,\,u}\left(B_{4}\right)$
has the trivial first row and column. Then, there is a matrix $B\in\Gamma_{3}$
such that
\begin{align*}
A & =\left(\begin{array}{cc}
1\\
 & B
\end{array}\right).
\end{align*}

Since the image of $\Gamma_{3}$ evaluated at $t=-1$ is isomorphic
to $\mathrm{SL}\left(2,\,\mathbb{Z}\right)$, just as the integral
Burau image of $B_{3}$, we may assume that $B$ is included in the
kernel of the evaluation at $t=-1$, by multiplying the image of a
braid in $B_{3}$. This kernel is isomoprhic to $F\times\mathbb{Z}$,
where $F$ is a free group and the infinite cyclic factor is generated
by $\beta_{3,\,u}\left(\sigma_{1}\sigma_{2}\sigma_{1}\right)^{4}$
(see \citep[Theorem 3.17]{lee2024salters}). If $B$ is a power of
$\beta_{3,\,u}\left(\sigma_{1}\sigma_{2}\sigma_{1}\right)^{4}$, we
establish the $\left(\mathrm{SH}_{4}\right)$. Suppose that $B$ is
not a power.

The braid Torelli group for $B_{4}$ is a free group of infinite rank,
normally generated by two elements $\left(\sigma_{1}\sigma_{2}\sigma_{1}\right)^{4}$
and $\left(\sigma_{3}\sigma_{2}\sigma_{3}\right)^{4}$ (see \citep[p. 316]{MR545151}).
Assuming the faithfulness of $\beta_{4,\,u}$, the matrix $A$ must
generate a free group of rank 2 together with $\beta_{4,\,u}\left(\sigma_{3}\sigma_{2}\sigma_{3}\right)^{4}$.
This leads to a contradiction, as $B$ must commute with the central
element $\beta_{3,\,u}\left(\sigma_{1}\sigma_{2}\sigma_{1}\right)^{4}$.
$\qed$

\noindent \end{prooftheorem11}
\begin{remark}

Long\textendash Paton \citep[Corollary 1.2]{MR1217079} and Dehornoy
\citep[Proposition 5, p. 11]{MR1381685} demonstrated that a matrix
$A\in\beta_{n,\,u}\left(B_{n}\right)$ has the trivial first row if
and only if $A$ has the trivial first column. In fact, a stronger
result holds: a matrix $A\in\beta_{n,\,u}\left(B_{n}\right)$ has
1 in the $\left(1,\,1\right)$-entry, then $A$ has the trivial first
row and column. Consequently, the premise of the separation hypotheses
$\left(\mathrm{SH}_{n}\right)$ can be simplified to requiring that
a matrix $A\in\beta_{n,\,u}\left(B_{n}\right)$ has 1 in the $\left(1,\,1\right)$-entry.

However, to avoid clutter, we do not prove this result in this paper.
Our proof relies on a general {*}-diagonalization of the Burau matrices
as outlined in \citep{lee2024salters}. For a simple case of $B_{4}$,
the proof hinges on analyzing the following functional equation:
\begin{align*}
f_{1}\overline{f_{1}}+\left(t^{-1}+t\right)f_{2}\overline{f_{2}}+\left(t^{-1}+t\right)\left(t^{-1}+1+t\right)f_{3}\overline{f_{3}} & =\left(t^{-1}+1+t\right)^{2},\:f_{1},\,f_{2},\,f_{3}\in\mathbb{Z}\left[t,\,t^{-1}\right],
\end{align*}
where $f_{1},\,f_{2},\,f_{3}$ are some linear combinations of the
entries in the first row of a Burau matrix $A\in\beta_{4,\,u}\left(B_{4}\right)$.
If a matrix $A$ has 1 in the $\left(1,\,1\right)$-entry, then by
the definition of $f_{1}$, we have $f_{1}\overline{f_{1}}=\left(t^{-1}+1+t\right)^{2}$.
This reduces the equation to
\begin{align*}
f_{2}\overline{f_{2}}+\left(t^{-1}+1+t\right)f_{3}\overline{f_{3}} & =0,\,f_{2},\,f_{3}\in\mathbb{Z}\left[t,\,t^{-1}\right].
\end{align*}

Evaluating this equation at infinitely many points $\zeta$, satisfying
that $\left(\overline{\zeta}+1+\zeta\right)>0$, on the unit circle
reveals that the only solution is $f_{2}=0=f_{3}$. This forces the
entire first row of $A$ to be trivial.

\noindent \end{remark}

\section{Counterexamples in the Stabilizer}

In this section, we consider the following stabilizer subgroup mentioned
in (2):
\begin{align*}
C_{B_{4}}\left(\sigma_{3}\right) & =\left\langle \sigma_{1},\,\sigma_{3},\,\left(\sigma_{3}\sigma_{2}\sigma_{3}\right)^{2}\right\rangle .
\end{align*}

Recall two braids $b_{1},\,b_{2}$ in $B_{4}$ were defined in Section
1 as
\begin{align*}
b_{1} & =\sigma_{1}\sigma_{3}^{-1},\:b_{2}=\sigma_{1}\sigma_{2}\sigma_{3}\sigma_{1}^{-1}\sigma_{2}^{-1}\sigma_{1}^{-1}.
\end{align*}

Gorin\textendash Lin \citep{MR251712} pointed out that these two
generate a free normal subgroup of $B_{4}$, and the quotient map
$\rho:B_{4}\to\left\langle \sigma_{2},\,\sigma_{3}\right\rangle $
induces a split short exact sequence, so that $B_{4}$ is a semidirect
product $\left\langle b_{1},\,b_{2}\right\rangle \rtimes\left\langle \sigma_{2},\,\sigma_{3}\right\rangle $.
Indeed, the following computation is direct from the braid relations
of $B_{4}$.
\begin{align}
 & \sigma_{2}b_{1}\sigma_{2}^{-1}=b_{2}^{-1}b_{1},\,\sigma_{3}b_{1}\sigma_{3}^{-1}=b_{1},\\
 & \sigma_{2}b_{2}\sigma_{2}^{-1}=b_{2},\,\sigma_{3}b_{2}\sigma_{3}^{-1}=b_{2}b_{1}.
\end{align}

\noindent \begin{lemma}

The subgroup $C_{B_{4}}\left(\sigma_{3}\right)$ is a semidirect product
$\left\langle b_{1},\,b_{2}b_{1}b_{2}^{-1}\right\rangle \rtimes\left\langle \sigma_{3},\,\left(\sigma_{3}\sigma_{2}\sigma_{3}\right)^{2}\right\rangle $
via the Gorin\textendash Lin quotient $\rho$.

\noindent \end{lemma}
\begin{proof}

Using the generating set (2), the image of $C_{B_{4}}\left(\sigma_{3}\right)$
under $\rho$ is $\left\langle \sigma_{3},\,\left(\sigma_{3}\sigma_{2}\sigma_{3}\right)^{2}\right\rangle $.
Since this group is isomorphic to $\mathbb{Z}^{2}$, the Reidemeister\textendash Schreier
process yields the generating set of the kernel
\begin{align*}
\left\langle \left(\sigma_{3}\sigma_{2}\sigma_{3}\right)^{2n}b_{1}\left(\sigma_{3}\sigma_{2}\sigma_{3}\right)^{-2n},\;n\in\mathbb{Z}\right\rangle ,
\end{align*}
where $\sigma_{3}$ commutes with $b_{1}$.

From equations (3) and (4), we derive the first two elements, distinct
from $b_{1}$, as follows:
\begin{align}
 & \left(\sigma_{3}\sigma_{2}\sigma_{3}\right)^{-2}b_{1}\left(\sigma_{3}\sigma_{2}\sigma_{3}\right)^{2}=b_{2}b_{1}^{-1}b_{2}^{-1},\\
 & \left(\sigma_{3}\sigma_{2}\sigma_{3}\right)^{2}b_{1}\left(\sigma_{3}\sigma_{2}\sigma_{3}\right)^{-2}=b_{1}^{-1}b_{2}b_{1}^{-1}b_{2}^{-1}b_{1}.
\end{align}

By combining equations (5) and (6), we obtain the following relation:
\begin{align}
\left(\sigma_{3}\sigma_{2}\sigma_{3}\right)^{2}b_{1}\left(\sigma_{3}\sigma_{2}\sigma_{3}\right)^{-2} & =b_{1}^{-1}\left(\left(\sigma_{3}\sigma_{2}\sigma_{3}\right)^{-2}b_{1}\left(\sigma_{3}\sigma_{2}\sigma_{3}\right)^{2}\right)b_{1}.
\end{align}

This relation (7) enables us to inductively express all other generators
as an appropriate product of $b_{1}$ and $b_{2}b_{1}b_{2}^{-1}$.
$\qedhere$

\noindent \end{proof}

Now we introduce a {*}-diagonazliation of the unitarity (1) for the
representation $\beta_{4,\,r}$. Define a matrix $M$ in $\mathrm{GL}\left(3,\,\mathbb{Z}\left[t,\,t^{-1},\,\left(1+t\right)^{-1}\right]\right)$
by
\begin{align*}
M & :=\left(\begin{array}{ccc}
0 & 1 & 0\\
1+t^{-1} & -t^{-1} & 0\\
0 & -t^{-1} & t^{-2}\left(1+t\right)
\end{array}\right).
\end{align*}

By conjugating $M$ to the images of generating braids, we have the
following matrices.
\begin{align*}
s_{1}:= & M\beta_{4,\,r}\left(\sigma_{1}\right)M^{-1}=\left(\begin{array}{ccc}
1 & 0 & 0\\
0 & -t & 0\\
0 & 0 & 1
\end{array}\right),\\
s_{2}:= & M\beta_{4,\,r}\left(\sigma_{2}\right)M^{-1}=\left(\begin{array}{ccc}
\frac{t\left(1-t\right)}{1+t} & \frac{t^{2}}{1+t} & \frac{t^{2}}{1+t}\\
\frac{1+t^{2}}{t\left(1+t\right)} & \frac{1}{1+t} & \frac{-t}{1+t}\\
\frac{1+t^{2}}{t\left(1+t\right)} & \frac{-t}{1+t} & \frac{1}{1+t}
\end{array}\right),\\
s_{3}:= & M\beta_{4,\,r}\left(\sigma_{3}\right)M^{-1}=\left(\begin{array}{ccc}
1 & 0 & 0\\
0 & 1 & 0\\
0 & 0 & -t
\end{array}\right).
\end{align*}

Define a diagonal matrix
\begin{align*}
D & :=\left(\begin{array}{ccc}
1 & 0 & 0\\
0 & t+t^{-1} & 0\\
0 & 0 & t+t^{-1}
\end{array}\right).
\end{align*}

By direct computation, for each matrix $A\in M\beta_{4,\,r}\left(B_{4}\right)M^{-1}$,
we have a unitarity relation:

\begin{equation}
\overline{A}DA^{T}=D.
\end{equation}

As can be seen in the case of $s_{2}$, the entire group $M\beta_{4,\,r}\left(B_{4}\right)M^{-1}$
has elements with $\left(1+t\right)$ in the denominator. This is
an obstacle when analyzing the structure of the unitary group. To
eliminate this, we introduce the following definition.

\noindent \begin{definition}

The \emph{tame subgroup} $\mathcal{T}\subset B_{4}$ is a set of braids
$b$ such that $M\beta_{4,\,r}\left(b\right)M^{-1}$ belongs to $\mathrm{GL}\left(3,\,\mathbb{Z}\left[t,\,t^{-1}\right]\right)$.

\noindent \end{definition}

For example, the braids $\sigma_{1}$ and $\sigma_{3}$ belong to
$\mathcal{T}$, while $\sigma_{2}$ does not. On the other hand, we
ask whether an element belonging to $\mathcal{T}$ remains in $\mathcal{T}$
after the conjugation by any braids. This is not always the case;
a direct computation yields $\sigma_{2}\sigma_{1}\sigma_{2}^{-1}\notin\mathcal{T}$.
This motivates the introduction of a new definition.

\noindent \begin{definition}

The \emph{stable subgroup} $\mathcal{S}\subset\mathcal{T}$ is a set
of braids $b$ such that $\sigma_{2}b\sigma_{2}^{-1}\in\mathcal{T}$.

\noindent \end{definition}

The subgroup $\mathcal{S}$ includes nontrivial examples; we have
\begin{align}
\left(s_{3}s_{2}s_{3}\right)^{2} & =\left(\begin{array}{ccc}
t\left(1-t+t^{2}\right) & t^{2}\left(1-t\right) & 0\\
\left(1-t\right)\left(t^{-1}+t\right) & 1-t+t^{2} & 0\\
0 & 0 & t^{3}
\end{array}\right),
\end{align}
which implies $\left(\sigma_{3}\sigma_{2}\sigma_{3}\right)^{2}\in\mathcal{S}$.
Moreover, define two matrices
\begin{align*}
S' & :=MSM^{-1}=M\beta_{4,\,r}\left(b_{1}\right)M^{-1},\\
T' & :=MTM^{-1}=M\beta_{4,\,r}\left(b_{2}\right)M^{-1}.
\end{align*}
\begin{lemma}

The tame subgroup $\mathcal{T}$ is $\left\langle b_{1},\,b_{2}\right\rangle \rtimes\left\langle \sigma_{3},\,\left(\sigma_{3}\sigma_{2}\sigma_{3}\right)^{2}\right\rangle $.
Moreover, the stable subgroup $\mathcal{S}$ is $\left\langle b_{1},\,b_{2}\right\rangle \rtimes\left\langle \left(\sigma_{3}\sigma_{2}\sigma_{3}\right)^{2}\right\rangle $.

\noindent \end{lemma}
\begin{proof}

\noindent A direct computation yields
\begin{align}
S' & =\left(\begin{array}{ccc}
1 & 0 & 0\\
0 & -t & 0\\
0 & 0 & -t^{-1}
\end{array}\right),\;T':=\left(\begin{array}{ccc}
-t^{-1}+1-t & 0 & t\left(1-t\right)\\
\left(t^{-1}-1\right)\left(t^{-1}+t\right) & 0 & -1+t-t^{2}\\
0 & -t^{-1} & 0
\end{array}\right).
\end{align}

Since $\left\langle b_{1},\,b_{2}\right\rangle $ is normal in $B_{4}$,
we see that $\mathcal{S}$ contains $\left\langle b_{1},\,b_{2}\right\rangle $.
By the semidirect product structure on $B_{4}$, it suffices to check
which of the braids in $\left\langle \sigma_{2},\,\sigma_{3}\right\rangle $
belong to $\mathcal{T}$ or $\mathcal{S}$. As the central quotient
of $B_{3}$ is isomorphic to the modular group $\mathrm{PSL}\left(2,\,\mathbb{Z}\right)\cong C_{2}*C_{3}$,
where $C_{n}$ is the cyclic group of order $n$, and as the generator
$\left(\sigma_{3}\sigma_{2}\sigma_{3}\right)^{2}$ of the center $Z\left(\left\langle \sigma_{2},\,\sigma_{3}\right\rangle \right)$
belongs to $\mathcal{S}$ as in (9), we only need to examine the elements
\begin{align*}
\left(\sigma_{3}\sigma_{2}\sigma_{3}\right)^{i_{1}}\left(\sigma_{3}\sigma_{2}\right)^{j_{1}}\left(\sigma_{3}\sigma_{2}\sigma_{3}\right)^{i_{2}}\cdots\left(\sigma_{3}\sigma_{2}\right)^{j_{n}},
\end{align*}
according to the normal form of the free product. The proof is completed
by using induction based on the length $n$. $\qedhere$

\noindent \end{proof}

In conclusion, we observe that many braids have integer coefficients
even in the {*}-diagonalized representation. Specifically, combining
Lemma 3.1 and Lemma 3.4, we conclude that the centralizer $C_{B_{4}}\left(\sigma_{3}\right)$
is contained in the tame subgroup $\mathcal{T}$.

Let us momentarily set aside the counterexamples and focus on the
faithfulness problem in the centralizer. For the usual Burau representation
$\beta_{4,\,r}$, since $\beta_{4,\,r}$ commutes with the Gorin\textendash Lin
quotient $\rho$ \citep{MR0375281}, it suffices to show that $\left\langle S',\,T'S'T'^{-1}\right\rangle $
is free of rank 2. This result was established by Moran in \citep{MR1113420},
but he used the order on the real field $\mathbb{R}$, which creates
challenges when extending the result to the mod $p$ Burau, where
the coefficient field has positive characteristic. Moreover, in general,
the Gorin\textendash Lin quotient $\rho$ does not necessarily commute
with $\beta_{4,\,r}$ modulo $p$ (for instance, see the counterexamples
in \citep{MR1431138}).

Therefore, we will rely here on the structure theory of $2\times2$
similitude groups over the Laurent polynomial ring over general fields,
as developed in \citep{lee2024buraurepresentationb3modulo}. In brief,
we consider the following groups and their generators.

\noindent \begin{definition}

Let $\mathbb{F}$ be a field. Define a diagonal matrix $D_{2}$ in
$\mathrm{PGL}\left(2,\,\mathbb{F}\left(t\right)\right)$ as

\[
D_{2}:=\left[\begin{array}{cc}
1 & 0\\
0 & t^{-1}+t
\end{array}\right].
\]

For simplicity, put $\Phi=t^{-1}+t$. The \emph{projective similitude
group} (with respect to $D_{2}$) $\mathrm{PS}\left(\mathbb{F}\right)$
is defined to be
\begin{align*}
\left\{ A\in\mathrm{PGL}\left(2,\,\mathbb{F}\left[t,\,t^{-1}\right]\right)\::\:\overline{A}D_{2}A^{T}=kD_{2}\,\mathrm{where}\;k\in\mathbb{F}^{\times}\right\} .
\end{align*}

On the other hand, the \emph{quaternionic group} $\mathrm{Q}\left(\mathbb{F}\right)$
is defined to be
\begin{align*}
\left\{ A\in\mathrm{PGL}\left(2,\,\mathbb{F}\left[t,\,t^{-1}\right]\right)\::\:A=\left[\begin{array}{cc}
g_{1} & g_{2}\\
-\Phi\overline{g_{2}} & \overline{g_{1}}
\end{array}\right],\,\mathrm{where}\;g_{1},\,g_{2}\in\mathbb{F}\left[t,\,t^{-1}\right]\right\} .
\end{align*}

The \emph{elementary generators} for $\mathrm{Q}\left(\mathbb{F}\right)$
are defined to be
\[
g_{\mathbb{F}}\left[r\right]\,:=\,\left[\begin{array}{cc}
t-r^{2} & r\\
-r\Phi & t^{-1}-r^{2}
\end{array}\right],
\]
for any $r\in\mathbb{F}$ such that $r^{4}\ne-1$.

On the other hand, suppose that $\mathbb{F}$ has characteristic 2.
For a polynomial $f\!\left[x\right]\in\mathbb{F}\left[x\right]$,
the \emph{upper additive generator} for $f$ is defined to be
\[
au_{\mathbb{F}}\left[f\right]\,:=\,\left[\begin{array}{cc}
1+\left(t^{-1}+t\right)f\!\left[t^{-1}+t\right] & \left(t^{-1}+1\right)f\!\left[t^{-1}+t\right]\\
\left(1+t\right)\Phi f\!\left[t^{-1}+t\right] & 1+\left(t^{-1}+t\right)f\!\left[t^{-1}+t\right]
\end{array}\right],
\]

while the \emph{lower additive generator} for $f$ is defined to be
\[
al_{\mathbb{F}}\left[f\right]\,:=\,\left[\begin{array}{cc}
1+\left(t^{-1}+t\right)f\!\left[t^{-1}+t\right] & \left(1+t\right)f\!\left[t^{-1}+t\right]\\
\left(t^{-1}+1\right)\Phi f\!\left[t^{-1}+t\right] & 1+\left(t^{-1}+t\right)f\!\left[t^{-1}+t\right]
\end{array}\right].
\]

\noindent \end{definition}
\begin{theorem}

For a field $\mathbb{F}$, the projective similitude group $\mathrm{PS}\left(\mathbb{F}\right)$
has the quaternionic group $\mathrm{Q}\left(\mathbb{F}\right)$ as
an index 4 or 2 subgroup with the coset representatives:

\[
\left[\begin{array}{cc}
1 & 0\\
0 & 1
\end{array}\right],\:\left[\begin{array}{cc}
1 & 0\\
0 & -1
\end{array}\right],\:\left[\begin{array}{cc}
1 & 0\\
0 & t
\end{array}\right],\:\left[\begin{array}{cc}
1 & 0\\
0 & -t
\end{array}\right].
\]

The quaternionic group $\mathrm{Q}\left(\mathbb{F}\right)$ is the
free product of groups $\left\{ \mathcal{G}_{\mathbb{F},\,r}\,:\,r\in\mathbb{F}\right\} $.
When $r\in\mathbb{F}$ satisfies $r^{4}\ne-1$, the group $\mathcal{G}_{\mathbb{F},\,r}$
is the infinite cyclic group generated by the corresponding elementary
generator $g_{\mathbb{F}}\left[r\right]$. In particular, if the equation
$x^{4}=-1$ has no solution over $\mathbb{F}$, the group $\mathrm{Q}\left(\mathbb{F}\right)$
is free of rank $\left|\mathbb{F}\right|$.

When $\mathbb{F}$ has characteristic 2, the element $1$ is the only
solution for the equation $x^{4}=-1$ with multiplicity 4. In this
case, the group $\mathcal{G}_{\mathbb{F},\,1}$ is isomorphic to the
free product of two infinitely generated abelian groups $\mathbb{F}\left[x\right]*\mathbb{F}\left[x\right]$,
where $\mathbb{F}\left[x\right]$ is the additive group of the polynomial
ring. The isomorphism is given by $f\mapsto au_{\mathbb{F}}\left[f\right]$
for the left factor and $f\mapsto al_{\mathbb{F}}\left[f\right]$
for the right factor. In particular, when $\mathbb{F}$ is the finite
field $\mathbb{F}_{2}$, then $\mathrm{Q}\left(\mathbb{F}\right)$
is isomorphic to $\mathbb{Z}*\mathbb{F}_{2}\left[x\right]*\mathbb{F}_{2}\left[x\right]$.

\noindent \end{theorem}
\begin{proof}

\noindent See \citep[Theorem 4.9]{lee2024salters} and \citep[Theorem 3.11]{lee2024buraurepresentationb3modulo}.
The only difference is that the palindromic polynomial $\Phi$ is
changed from $t^{-1}+1+t$ to $t^{-1}+t$, which makes only minor
mechanical differences to all arguments. $\qedhere$

\noindent \end{proof}

In order to prove the faithfulness, we need more detailed information
about the the whole similitude group. Define two matrices
\begin{align*}
h_{0} & :=\left[\begin{array}{cc}
1 & 0\\
0 & -t
\end{array}\right],\;h_{-1}:=\left[\begin{array}{cc}
1 & 0\\
0 & -1
\end{array}\right].
\end{align*}
\begin{lemma}

Let $\mathbb{F}$ be a field. Then, the projective similitude group
$\mathrm{PS}\left(\mathbb{F}\right)$ is a semidirect product $\left\langle \mathrm{Q}\left(\mathbb{F}\right),\,h_{0}\right\rangle \rtimes\left\langle h_{-1}\right\rangle $.
The element $h_{-1}$ acts by
\begin{align}
h_{-1}h_{0}h_{-1}^{-1} & =h_{0},\:h_{-1}g_{\mathbb{F}}\left[r\right]h_{-1}^{-1}=g_{\mathbb{F}}\left[-r\right],\:r\in\mathbb{F}.
\end{align}

The quaternionic group $\mathrm{Q}\left(\mathbb{F}\right)$ is an
index 2 subgroup of $\left\langle \mathrm{Q}\left(\mathbb{F}\right),\,h_{0}\right\rangle $.
In the group presentation of $\left\langle \mathrm{Q}\left(\mathbb{F}\right),\,h_{0}\right\rangle $,
the generator $h_{0}$ is added, one elementary generator $g_{\mathbb{F}}\left[0\right]$
is deleted by a relation $h_{0}^{2}=g_{\mathbb{F}}\left[0\right]^{-1}$.
On the elementary generators, for each $r\in\mathbb{F}$ such that
$r\ne0$, we add exactly one relation:
\begin{align}
h_{0}g_{\mathbb{F}}\left[r\right]h_{0}^{-1} & =g_{\mathbb{F}}\left[-r^{-1}\right]^{-1}h_{0}^{-2}.
\end{align}

Moreover, suppose $\mathrm{char}\left(\mathbb{F}\right)=2$. For a
polynomial $f\!\left[x\right]\in\mathbb{F}\left[x\right]$, $h_{0}$
acts on the lower additive generators
\begin{equation}
h_{0}al_{\mathbb{F}}\left(f\right)h_{0}^{-1}=au_{\mathbb{F}}\left(f\right).
\end{equation}
\end{lemma}
\begin{proof}

\noindent The equations (11), (12) and (13) are direct from matrix
calculations. Moreover, we have
\begin{align*}
h_{0}^{2} & =\left[\begin{array}{cc}
1 & 0\\
0 & t^{2}
\end{array}\right]=\left[\begin{array}{cc}
t^{-1} & 0\\
0 & t
\end{array}\right]=g_{\mathbb{F}}\left[0\right]^{-1}.
\end{align*}

There are additional relations for type groups $\mathcal{G}_{\mathbb{F},\,r}$
such that $r$ is a root with multiplicity 1, but we omit them here
to avoid clutter. $\qedhere$

\noindent \end{proof}
\begin{prooftheorem13}

At first, suppose the field $\mathbb{F}$ is the rational $\mathbb{Q}$
(usual Burau) or the finite field $\mathbb{F}_{p}$ for $p>2$. Then,
the determinant map from the centralizer $\beta_{4,\,r}\left(C_{B_{4}}\left(\sigma_{3}\right)\right)$
to the units $\left\langle \left(-t\right)^{n}\right\rangle \subset\mathbb{F}\left[t,\,t^{-1}\right]^{\times}$
splits, which implies
\[
M\beta_{4,\,r}\left(C_{B_{4}}\left(\sigma_{3}\right)\right)M^{-1}=\left\langle S',\,\left(s_{2}s_{3}s_{2}\right)^{2}s_{3}^{-6}\right\rangle \times\left\langle s_{3}\right\rangle .
\]

From (9) and (10), consider the following map:
\[
\left(\begin{array}{ccc}
a_{11} & a_{12} & 0\\
a_{21} & a_{22} & 0\\
0 & 0 & *
\end{array}\right)\mapsto\left(\begin{array}{cc}
a_{11} & a_{12}\\
a_{21} & a_{22}
\end{array}\right),
\]

\noindent and let $c$ (resp. $d$) be the image of $S'$ (resp. $\left(s_{2}s_{3}s_{2}\right)^{2}s_{3}^{-6}$)
in $\mathrm{GL}\left(2,\,\mathbb{F}\left[t,\,t^{-1}\right]\right)$.
Concretely, we have
\[
c=\left(\begin{array}{cc}
1 & 0\\
0 & -t
\end{array}\right),\:d=\left(\begin{array}{cc}
t\left(1-t+t^{2}\right) & t^{2}\left(1-t\right)\\
\left(1-t\right)\left(t^{-1}+t\right) & 1-t+t^{2}
\end{array}\right).
\]
\begin{claim1}

Suppose $\mathrm{char}\left(\mathbb{F}\right)\ne2$. In $\mathrm{GL}\left(2,\,\mathbb{F}\left[t,\,t^{-1}\right]\right)$,
take the lifts $\widetilde{h_{0}}=\left(\begin{array}{cc}
1 & 0\\
0 & -t
\end{array}\right)$ of $h_{0}$, $\widetilde{h_{-1}}=\left(\begin{array}{cc}
1 & 0\\
0 & -1
\end{array}\right)$ of $h_{-1}$, and $\widetilde{g_{\mathbb{F}}}\left[1\right]=\left(\begin{array}{cc}
t-1 & 1\\
-\Phi & t^{-1}-1
\end{array}\right)$ of $g_{\mathbb{F}}\left[1\right]$. Define two matrices
\begin{align*}
e_{-2t} & :=\left(\begin{array}{cc}
-2t & 0\\
0 & -2t
\end{array}\right),\:e_{-4}:=\left(\begin{array}{cc}
-4 & 0\\
0 & -4
\end{array}\right).
\end{align*}

Then, the group $G=\left\langle e_{-2t},\,e_{-4},\,\widetilde{h_{0}},\,\widetilde{h_{-1}},\,\widetilde{g_{\mathbb{F}}}\left[1\right]\right\rangle $
has a structure as 
\begin{align}
\left(\left(\left\langle e_{-2t}\right\rangle \times\left\langle \widetilde{h_{0}},\,\widetilde{g_{\mathbb{F}}}\left[1\right]\right\rangle \right)\rtimes\left\langle \widetilde{h_{-1}}\right\rangle \right)\times\left\langle e_{-4}\right\rangle ,
\end{align}
where $\left\langle e_{-2t}\right\rangle \cong\mathbb{Z}$, $\left\langle \widetilde{h_{0}},\,\widetilde{g_{\mathbb{F}}}\left[1\right]\right\rangle $
is free of rank 2, $\left\langle \widetilde{h_{-1}}\right\rangle \cong C_{2}$,
and $\left\langle e_{-4}\right\rangle \cong\left\langle -4\right\rangle $,
where ${-4\in\mathbb{F}^{\times}}$. The element $\widetilde{h_{-1}}$
commutes with $e_{-2t}$, $e_{-4}$ and $\widetilde{h_{0}}$, but
nontrivially acts on $\widetilde{g_{\mathbb{F}}}\left[1\right]$ by
\begin{align}
\widetilde{h_{-1}}\widetilde{g_{\mathbb{F}}}\left[1\right]\widetilde{h_{-1}}^{-1} & =e_{-2t}\left(h_{0}^{-1}\widetilde{g_{\mathbb{F}}}\left[1\right]^{-1}h_{0}^{-1}\right).
\end{align}
\end{claim1}
\begin{proofclaim1}

\noindent By (12), the group $\left\langle h_{0},\,g_{\mathbb{F}}\left[1\right]\right\rangle $
is free of rank 2 by deleting another elementary generator $g_{\mathbb{F}}\left[-1\right]$.
Since $e_{-2t}$ and $e_{-4}$ belong to the center of the general
linear group, the group presentation of $G$ is given by expressing
the lifted relators as products of $e_{-2t}$ and $e_{-4}$. The only
difference comes from (15). $\qed$

\noindent \end{proofclaim1}

Using the matrices defined in Claim 1, we obtain the following expression:
\begin{align*}
c & =\widetilde{h_{0}},\:d=\left(e_{-2t}e_{-4}^{-1}\right)\widetilde{h_{-1}}\widetilde{h_{0}}^{-1}\widetilde{g_{\mathbb{F}}}\left[1\right]\widetilde{h_{0}}\widetilde{g_{\mathbb{F}}}\left[1\right]\widetilde{h_{0}}.
\end{align*}

By (14), the group $\left\langle c,\,d\right\rangle $ has an index
2 subgroup $\left\langle c,\,d^{2},\,d^{-1}cd\right\rangle $. From
(15), we have
\begin{align*}
d^{2} & =\left(e_{-2t}^{4}e_{-4}^{-2}\right)\widetilde{h_{0}}^{-2}\widetilde{g_{\mathbb{F}}}\left[1\right]^{-1}\widetilde{h_{0}}^{-1}\widetilde{g_{\mathbb{F}}}\left[1\right]^{-1}\widetilde{h_{0}}^{-1}\widetilde{g_{\mathbb{F}}}\left[1\right]\widetilde{h_{0}}\widetilde{g_{\mathbb{F}}}\left[1\right]\widetilde{h_{0}},\\
d^{-1}cd & =\widetilde{h_{0}}^{-1}\widetilde{g_{\mathbb{F}}}\left[1\right]^{-1}\widetilde{h_{0}}^{-1}\widetilde{g_{\mathbb{F}}}\left[1\right]^{-1}\widetilde{h_{0}}\widetilde{g_{\mathbb{F}}}\left[1\right]\widetilde{h_{0}}\widetilde{g_{\mathbb{F}}}\left[1\right]\widetilde{h_{0}},
\end{align*}

\noindent where we have
\begin{align}
d^{2} & =\left(e_{-2t}^{4}e_{-4}^{-2}\right)c^{-1}\left(d^{-1}cd\right)^{-1}.
\end{align}

From (16), we have $\left\langle c,\,d^{2},\,d^{-1}cd\right\rangle =\left\langle c,\,d^{-1}cd\right\rangle \times\left\langle e_{-2t}^{4}e_{-4}^{-2}\right\rangle $,
where $\left\langle c,\,d^{-1}cd\right\rangle $ is free from the
freeness of $\left\langle \widetilde{h_{0}},\,\widetilde{g_{\mathbb{F}}}\left[1\right]\right\rangle $.
For simplicity, put $e=e_{-2t}^{4}e_{-4}^{-2}$. We derive the presentation
of $\left\langle c,\,d\right\rangle =\left\langle c,\,d^{-1}cd,\,e,\,d\right\rangle $
by listing all relations as follows:
\begin{align*}
1 & =\left[e,\,c\right]=\left[e,\,d^{-1}cd\right]=\left[e,\,d\right],\\
d^{2} & =ec^{-1}\left(d^{-1}cd\right)^{-1},\\
c & =d\left(d^{-1}cd\right)d^{-1},\\
dcd^{-1} & =c^{-1}\left(d^{-1}cd\right)c,
\end{align*}
where $\left[x,\,y\right]=x^{-1}y^{-1}xy$ means the usual commutator
and the final equality follows from (7). By deleting two generators
$d^{-1}cd$ and $e$ from the two middle rows, the only remaining
generators are $c,\,d$, as expected. The remaining relations are
\begin{align*}
1 & =\left[dcdc,\,c\right]=\left[dcdc,\,d\right],\\
dcd^{-1} & =c^{-1}\left(d^{-1}cd\right)c.
\end{align*}

Finally, the first row $1=\left[dcdc,\,c\right]=\left[dcdc,\,d\right]$
follows from the last one as follows:
\begin{align*}
\left[dcdc,\,c\right] & =c^{-1}d^{-1}c^{-1}\left(d^{-1}c^{-1}d\right)cdc^{2}\\
 & =c^{-1}d^{-1}c^{-1}\left(cdc^{-1}d^{-1}c^{-1}\right)cdc^{2}\\
 & =1,
\end{align*}
where the second eqaulity follows from the relation $dcd^{-1}=c^{-1}\left(d^{-1}cd\right)c$,
and
\begin{align*}
\left[dcdc,\,d\right] & =c^{-1}d^{-1}c^{-1}d^{-1}d^{-1}dcdcd\\
 & =c^{-1}d^{-1}c^{-1}\left(d^{-1}cd\right)cd\\
 & =c^{-1}d^{-1}c^{-1}\left(cdcd^{-1}c^{-1}\right)cd\\
 & =1,
\end{align*}
where the third equality follows from the relation $dcd^{-1}=c^{-1}\left(d^{-1}cd\right)c$.
Since linear groups are Hopfian, the proof is complete from Lemma
3.1 and (7).

On the other hand, suppose $\mathbb{F}=\mathbb{F}_{2}$ and consider
the mod 2 Burau images.

\noindent \begin{claim2}

In $\mathrm{GL}\left(2,\,\mathbb{F}_{2}\left[t,\,t^{-1}\right]\right)$,
take the lifts $\widetilde{h_{0}}=\left(\begin{array}{cc}
1 & 0\\
0 & t
\end{array}\right)$ of $h_{0}$ and
\begin{align*}
\widetilde{au_{\mathbb{F}}}\left[1\right] & =\left(\begin{array}{cc}
1+t^{-1}+t & \left(t^{-1}+1\right)\\
\left(1+t\right)\Phi & 1+t^{-1}+t
\end{array}\right)
\end{align*}
of $au_{\mathbb{F}}\left[1\right]$. Define a matrix $e_{t^{2}}:=\left(\begin{array}{cc}
t^{2} & 0\\
0 & t^{2}
\end{array}\right)$. Then, the group $G_{2}=\left\langle e_{t^{2}},\,\widetilde{h_{0}},\,\widetilde{au_{\mathbb{F}}}\left[1\right]\right\rangle $
has a structure as $h_{0}\,al_{\mathbb{F}}\left(f\right)\,h_{0}^{-1}=au_{\mathbb{F}}\left(f\right).$
\begin{align}
\left(\left\langle \widetilde{au_{\mathbb{F}}}\left[1\right]\right\rangle *\left\langle \widetilde{h_{0}}\right\rangle \right)\times\left\langle e_{t^{2}}\right\rangle ,
\end{align}
where $\left\langle \widetilde{au_{\mathbb{F}}}\left[1\right]\right\rangle \cong C_{2}$
and $\left\langle \widetilde{h_{0}}\right\rangle \cong\mathbb{Z}\cong\left\langle e_{t^{2}}\right\rangle $.\end{claim2}
\begin{proofclaim2}

\noindent By (13), the group $\left\langle h_{0},\,au_{\mathbb{F}}\left[1\right]\right\rangle $
is a free product $\left\langle au_{\mathbb{F}}\left[1\right]\right\rangle *\left\langle h_{0}\right\rangle $
by deleting the lower additive generator $al_{\mathbb{F}}\left[1\right]$.
Since $e_{t^{2}}$ belongs to the center of the general linear group,
the group presentation of $G_{2}$ is given by expressing the lifted
relators as product of $e_{t^{2}}$. The only relator $\widetilde{au_{\mathbb{F}}}\left[1\right]^{2}$
is trivial. $\qed$

\noindent \end{proofclaim2}

Using the matrices defined in Claim 2, we obtain the following expression:
\begin{align*}
c & =\widetilde{h_{0}},\:d=e_{t^{2}}\widetilde{h_{0}}^{-2}\widetilde{au_{\mathbb{F}}}\left[1\right]\widetilde{h_{0}}.
\end{align*}

Put $e'=c^{2}dc^{-1}$. By (17), the group $\left\langle c,\,d\right\rangle =\left\langle c,\,e'\right\rangle $
has an index 2 subgroup $\left\langle c,\,e'^{2},\,e'^{-1}ce'\right\rangle $,
where
\begin{align*}
e'^{2} & =e_{t^{2}}^{2},\\
e'^{-1}ce' & =\widetilde{au_{\mathbb{F}}}\left[1\right]\widetilde{h_{0}}\,\widetilde{au_{\mathbb{F}}}\left[1\right],
\end{align*}
where we have $\left\langle c,\,e'^{2},\,e'^{-1}ce'\right\rangle =\left(\left\langle c\right\rangle *\left\langle e'^{-1}ce'\right\rangle \right)\times\left\langle e'^{2}\right\rangle $
from (17). We derive the presentation of $\left\langle c,\,e'\right\rangle =\left\langle c,\,e'^{-1}ce',\,e'^{2},\,e'\right\rangle $
by listing all relations as follows
\begin{align*}
1 & =\left[e'^{2},\,c\right]=\left[e'^{2},\,e'^{-1}ce'\right],\\
e'ce'^{-1} & =e'^{-1}ce',
\end{align*}
where we trivially deleted two generators $e'^{-1}ce',\,e'^{2}$.
However, the last row is simplified as $1=\left[e'^{2},\,c\right]$,
and the first two relations are immediately redundant.

In summary, the group $\left\langle c,\,d\right\rangle $ has two
generators $c,\,d$ and one relation $1=\left[c^{2}dcdc^{-1},\,c\right]$.
The right-hand side is $cd^{-1}c^{-1}d^{-1}c^{-1}dcd$, and we can
reorganize the relation as follows:
\begin{align*}
dcd^{-1} & =c^{-1}\left(d^{-1}cd\right)c,
\end{align*}
which concludes the proof of Theorem 1.3. $\qed$

\noindent \end{prooftheorem13}
\begin{remark}

We may use Theorem 1.3 to prove the partial faithfulness of $\beta_{4,\,r}$
or $\beta_{4,\,r}$ modulo $p$ in the centralizer of other braids.
A good example is $C_{B_{4}}\left(\sigma_{2}\right)$; we directly
compute a generating set of this group from (2), from $\left(\sigma_{1}\sigma_{2}\sigma_{3}\right)\sigma_{i}\left(\sigma_{1}\sigma_{2}\sigma_{3}\right)^{-1}=\sigma_{i+1}$
for $i=1,2$ and $\left(\sigma_{1}\sigma_{2}\sigma_{3}\right)\sigma_{3}\left(\sigma_{1}\sigma_{2}\sigma_{3}\right)^{-1}=b_{2}\sigma_{2}$.
It yields

\[
C_{B_{4}}\left(\sigma_{2}\right)=\left\langle b_{2},\,b_{1}^{-1}b_{2}b_{1}\right\rangle \rtimes\left\langle \sigma_{2},\,\left(\sigma_{1}\sigma_{2}\sigma_{1}\right)^{2}\right\rangle .
\]

Since the representations are faithful, we also have the freeness
of $\left\langle T,\,S^{-1}TS\right\rangle $. Moreover, since $\sigma_{2}$
commutes with both $\left(\sigma_{1}\sigma_{2}\sigma_{1}\right)^{2}$
and $\left(\sigma_{3}\sigma_{2}\sigma_{3}\right)^{2}$, we extend
the classical result of Smythe \citep{MR545151} that the Burau images
of these two braids generate a free group of rank 2, to the Burau
images modulo $p$.

\noindent \end{remark}

Now we return to Salter's question. To answer it, we need to construct
a counterexample in the formally defined group $\overline{\Gamma_{4}'}$.
Recall that we defined a matrix $M\in\mathrm{GL}\left(3,\,\mathbb{Z}\left[t,\,t^{-1},\,\left(1+t\right)^{-1}\right]\right)$
in the beginning of Section 3.

\noindent \begin{lemma}

Suppose $A\in\mathrm{GL}\left(3,\,\mathbb{Z}\left[t,\,t^{-1}\right]\right)$
such that
\begin{align*}
A & =\left(\begin{array}{ccc}
A_{11} & A_{12} & A_{13}\\
A_{21} & A_{22} & A_{23}\\
A_{31} & A_{32} & A_{33}
\end{array}\right).
\end{align*}

Then, we have the following facts about $A$.
\begin{enumerate}
\item $MAM^{-1}\in\mathrm{GL}\left(3,\,\mathbb{Z}\left[t,\,t^{-1}\right]\right)$
if and only if $A_{21}|_{t=-1}=0=A_{23}|_{t=-1}$.
\item $Ms_{2}As_{2}^{-1}M^{-1}\in\mathrm{GL}\left(3,\,\mathbb{Z}\left[t,\,t^{-1}\right]\right)$
if and only if $\left(1,\,-1,\,-1\right)^{T}$ is an eigenvector of
$A^{T}|_{t=-1}$.
\item $M^{-1}AM\in\mathrm{GL}\left(3,\,\mathbb{Z}\left[t,\,t^{-1}\right]\right)$
if and only if $\left(1,\,1,\,1\right)^{T}$ is an eigenvector of
$A|_{t=-1}$.
\end{enumerate}
\noindent \end{lemma}
\begin{proof}

\noindent Direct computation. $\qedhere$

\noindent \end{proof}

We are ready to prove Theorem 1.2. Based on our analysis of the centralizer
of $\sigma_{3}$, we formulate the \emph{centralizer hypothesis} and
the \emph{weak centralizer hypothesis}.

\noindent \begin{hypothesis}

$ $
\begin{itemize}
\item $\left(\mathrm{CH}\right)$ $\beta_{4,\,r}\left(C_{B_{4}}\left(\sigma_{3}\right)\right)=C_{\beta_{4,\,r}\left(B_{4}\right)}\left(\beta_{4,\,r}\left(\sigma_{3}\right)\right).$
\item $\left(\mathrm{WCH}\right)$ $C_{\beta_{4,\,r}\left(B_{4}\right)}\left(\beta_{4,\,r}\left(\sigma_{3}\right)\right)\cap\left\langle S,\,T\right\rangle =\left\langle S,\,TST^{-1}\right\rangle $.
\end{itemize}
\noindent \end{hypothesis}

It is obvious that the faithfulness of $\beta_{4,\,r}$ implies $\left(\mathrm{CH}\right)$
and $\left(\mathrm{CH}\right)$ implies $\left(\mathrm{WCH}\right)$.
In fact, for the usual Burau $\beta_{4,\,r}$, $\left(\mathrm{WCH}\right)$
again implies $\left(\mathrm{CH}\right)$, but we have introduced
$\left(\mathrm{WCH}\right)$ separately considering possible extensions
to the mod $p$ cases.

\noindent \begin{prooftheorem12}

We find a matrix $A\in\Gamma_{4}'$ with determinant 1 that commutes
with $\beta_{4,\,r}\left(\sigma_{3}\right)$. In the light of Lemma
3.9, it is convenient to consider a matrix $MAM^{-1}$ satisfying
the diagonalized unitarity (8). Suppose a matrix $B\in\mathrm{GL}\left(3,\,\mathbb{Z}\left[t,\,t^{-1}\right]\right)$
has the following form:
\begin{align*}
\left(\begin{array}{ccc}
B_{11} & B_{12} & 0\\
B_{21} & B_{22} & 0\\
0 & 0 & 1
\end{array}\right),
\end{align*}
satisfying $s_{3}B=Bs_{3}$. Put $B'=\left(\begin{array}{cc}
B_{11} & B_{12}\\
B_{21} & B_{22}
\end{array}\right)$. We find $B'$ making $B$ as a counterexample.

By using the tree-based algorithm described in \citep[Section 5]{lee2024salters},
there is an element in $\mathrm{PSL}\left(2,\,\mathbb{Z}\left[t,\,t^{-1}\right]\right)$:
\begin{align*}
g_{\mathbb{Q}}\left[-\frac{1}{2}\right]^{-1}g_{\mathbb{Q}}\left[\frac{6}{5}\right]g_{\mathbb{Q}}\left[-\frac{7}{13}\right]^{-1}g_{\mathbb{Q}}\left[\frac{13}{15}\right]g_{\mathbb{Q}}\left[-\frac{8}{13}\right]^{-1}g_{\mathbb{Q}}\left[\frac{5}{6}\right]g_{\mathbb{Q}}\left[-2\right]^{-1},
\end{align*}
and taking a lift $C$ of this in $\mathrm{SL}\left(2,\,\mathbb{Z}\left[t,\,t^{-1}\right]\right)$,
we have an element:
\begin{align*}
C & =\left(\begin{array}{cc}
\frac{\left(1-t+t^{2}\right)\left(-2+6t-9t^{2}+8t^{3}-6t^{4}+2t^{5}\right)}{t^{4}} & \frac{\left(1-t\right)\left(2-2t+t^{2}\right)\left(-2+2t-2t^{2}+t^{3}\right)}{t^{3}}\\
\frac{\left(-1+t\right)\left(1+t^{2}\right)\left(1-2t+2t^{2}\right)\left(-1+2t-2t^{2}+2t^{3}\right)}{t^{4}} & \frac{\left(1-t+t^{2}\right)\left(2-6t+8t^{2}-9t^{3}+6t^{4}-2t^{5}\right)}{t^{3}}
\end{array}\right).
\end{align*}

We can verify by hand that $\det\left(C\right)=1$. However, since
$\left(1,\,1\right)^{T}$ is not an eigenvector of $C|_{t=-1}$, we
need to modify this matrix. By using the matrix $\widetilde{h_{-1}}=\left(\begin{array}{cc}
1 & 0\\
0 & -1
\end{array}\right)$, we have an involution
\begin{align*}
\widetilde{h_{-1}}\left(\begin{array}{cc}
a_{11} & a_{12}\\
a_{21} & a_{22}
\end{array}\right)\widetilde{h_{-1}} & =\left(\begin{array}{cc}
a_{11} & -a_{12}\\
-a_{21} & a_{22}
\end{array}\right).
\end{align*}

Thus, we have
\begin{align*}
\left(C\widetilde{h_{-1}}C\widetilde{h_{-1}}\right)|_{t=-1} & =\left(\begin{array}{cc}
1 & 0\\
0 & 1
\end{array}\right).
\end{align*}

Choose $B'=C\widetilde{h_{-1}}C\widetilde{h_{-1}}$. By Lemma 3.9
(3), the group $\Gamma_{4}'$ includes the matrix $M^{-1}BM$ with
determinant 1. Suppose that the Burau image $\beta_{4,\,r}\left(B_{4}\right)$
includes $M^{-1}BM$. It is included in the Burau image of the tame
subgroup $\mathcal{T}$ by construction, but not included in the image
of the stable subgroup $\mathcal{S}$, satisfying
\begin{align*}
\left(M^{-1}BM\right)|_{t=-1} & =\left(\begin{array}{ccc}
1 & 41616 & 0\\
0 & 1 & 0\\
0 & -17238 & 1
\end{array}\right).
\end{align*}

Now choose a matrix $A$ as
\begin{align*}
A & =\left(M^{-1}BM\right)\beta_{4,\,r}\left(\sigma_{1}^{-58854}\left(\sigma_{3}\sigma_{2}\sigma_{3}\right)^{19618}\right).
\end{align*}

By Lemma 3.9 (2), the matrix $A$ is in the image of $\mathcal{S}$.
Since $\det A=1$, we have ${A\in\beta_{4,\,r}\left(b_{1},\,b_{2}\right)}$
from Lemma 3.4. Assume the weak centralizer hypothesis $\left(\mathrm{WCH}\right)$.
Then, the matrix $A$ is included in the group $\left\langle S,\,TST^{-1}\right\rangle $.
However, it is impossible, since the $2\times2$ submatrices in the
upper left of $MSM^{-1}$ and $MTST^{-1}M^{-1}$ are generated by
$h_{0}$ and $g_{\mathbb{Q}}\left[1\right]$ in the projective image,
as in the proof of Theorem 1.3. $\qed$

\noindent \end{prooftheorem12}
\begin{remark}

We may pose Salter's question \emph{modulo }$p$, by defining formal
matrix groups $\Gamma'$ over $\mathbb{F}_{p}\left[t,\,t^{-1}\right]$.
Cooper\textendash Long \citep{MR1431138} showed that the central
quotient of the mod 2 reduced Burau image of $B_{4}$ generates the
whole projective unitary group over $\mathbb{F}_{2}\left[t,\,t^{-1}\right]$.
Therefore, Salter's question of $B_{4}$ modulo 2 is positively solved.
The cases $p>2$ are more tricky, but if we assume $\left(\mathrm{WCH}\right)$
modulo $p$, the counterexample $A$ is a generic counterexample for
$p\ge5$. In brief, for $p\ne17$, the element $g_{\mathbb{F}_{p}}\left[-\frac{1}{2}\right]$
is not able to be generated by $\left\langle h_{0},\,g_{\mathbb{F}_{p}}\left[1\right]\right\rangle $.
For $p=17$, the element $g_{\mathbb{F}_{17}}\left[-\frac{7}{13}\right]$
is such an element. Furthermore, this counterexample $A$ shows that
the centralizer in the Burau image modulo $p\ge5$ is an infinite
index subgroup of the centralizer in Salter's target group $\Gamma'$.
Under $\left(\mathrm{WCH}\right)$, this explains why Cooper\textendash Long
\citep{MR1668343} was not able to algorithmically compute the index
of the mod 5 Burau image in the whole projective unitary group.

We may also consider the mod $p$ (or even mod $n$) tame subgroup
$\mathcal{T}_{p}$ and the mod $p$ stable subgroup $\mathcal{S}_{p}$
as subgroups of $B_{4}$. By definition, the group $\mathcal{T}_{p}$
(resp. $\mathcal{S}_{p}$) contains $\mathcal{T}$ (resp. $\mathcal{S}$).
In the light of Lemma 3.9 (1), the quotient $\mathcal{T}_{p}/\left\langle b_{1},\,b_{2}\right\rangle $
is isomorphic to the (upper) Hecke congruence subgroup of level $p$
modulo the center, as expected by \citep{lee2024buraurepresentationb3modulo}.

\noindent \end{remark}

\section{The Entire Unitary Group}

In this section, we consider the unitary group as a whole satisfying
the unitarity (8):
\begin{align*}
\overline{A}\left(\begin{array}{ccc}
1 & 0 & 0\\
0 & t+t^{-1} & 0\\
0 & 0 & t+t^{-1}
\end{array}\right)A^{T} & =\left(\begin{array}{ccc}
1 & 0 & 0\\
0 & t+t^{-1} & 0\\
0 & 0 & t+t^{-1}
\end{array}\right),
\end{align*}
where we recall that $D=\left(\begin{array}{ccc}
1 & 0 & 0\\
0 & t+t^{-1} & 0\\
0 & 0 & t+t^{-1}
\end{array}\right)$. Slightly abusing notation, we also denote its projectivization by
$D$.

\noindent \begin{definition}

For a field $\mathbb{F}$, the \emph{projective unitary group} (with
respect to $D$), $\mathrm{PU}\left(\mathbb{F}\right)$, is defined
to be
\begin{align*}
\left\{ A\in\mathrm{PGL}\left(3,\,\mathbb{F}\left[t,\,t^{-1}\right]\right)\::\:\overline{A}DA^{T}=D\right\} .
\end{align*}
\end{definition}

A natural approach would be to act this unitary group on the associated
Bruhat-Tits building. For general definitions of buildings, see \citep{MR2439729,MR3415244,MR1431138}.
In brief, for a field $\mathbb{F}$, we consider the two-dimensional
Bruhat-Tits building $X_{\mathbb{F}\left(t\right)}$ with the valuation
at infinity $v_{\infty}$, which is a simplicial complex whose vertices
are $3\times3$ equivalence classes of lattices along with a uniformizing
element $\pi=t^{-1}$. We consider the action by the left multiplication
$\mathrm{PGL}\left(3,\,\mathbb{F}\left[t,\,t^{-1}\right]\right)$
on $X_{\mathbb{F}\left(t\right)}$.

It appears difficult to obtain general results independent of the
base field $\mathbb{F}$, as in the case of the $2\times2$ similitude
group. Attempting to derive generators based on the elementary generators
$g_{\mathbb{F}}\left[r\right]$ leads to issues in the remaining dimension.
The (modulo square) determinant of an elementary generator $g_{\mathbb{F}}\left[r\right]$
is $1+r^{4}$ for each $r\in\mathbb{F}$, but in many fields, the
square root $\sqrt{1+r^{4}}$ does not exist in general. When working
with fields of characteristic 0, for computational purposes, we must
restrict our focus to cases where a square root is defined for a broad
range of elements. A natural choice of such fields is the reals $\mathbb{R}$,
or more generally a Euclidean field $\mathbb{E}$, which is an ordered
field in which there is a square root for every non-negative element
(see \citep{MR2104929}).

From now on, for a matrix $A\in\mathrm{PGL}\left(3,\,\mathbb{E}\left(t\right)\right)$,
denote by $\left[A\right]$ the lattice $A\cdot\left[I\right]$, where
$\left[I\right]$ is the standard lattice. We sometimes refer to $\left[I\right]$
as the \emph{identity lattice}.

\noindent \begin{definition}

For a Euclidean field $\mathbb{E}$, for each element $r\in\mathbb{E}$
such that $0<r$, the \emph{elementary generator} $k_{\mathbb{E}}\left[r\right]$
in $\mathrm{PU}\left(\mathbb{E}\right)$ is defined to be
\begin{align*}
k_{\mathbb{E}}\left[r\right]\, & :=\,\left[\begin{array}{ccc}
\frac{r^{2}-t}{\sqrt{1+r^{4}}} & \frac{-rt}{\sqrt{1+r^{4}}} & 0\\
\frac{-r\left(t^{-1}+t\right)}{\sqrt{1+r^{4}}} & \frac{1-r^{2}t}{\sqrt{1+r^{4}}} & 0\\
0 & 0 & 1
\end{array}\right].
\end{align*}

We add two \emph{elementary generators} $k_{\mathbb{E}}\left[0\right]$
and $k_{\mathbb{E}}\left[\infty\right]$ as follows.
\begin{align*}
k_{\mathbb{E}}\left[0\right] & :=\left[\begin{array}{ccc}
-t & 0 & 0\\
0 & 1 & 0\\
0 & 0 & 1
\end{array}\right],\:k_{\mathbb{E}}\left[\infty\right]:=\left[\begin{array}{ccc}
1 & 0 & 0\\
0 & -t & 0\\
0 & 0 & 1
\end{array}\right].
\end{align*}

Finally, for each orthogonal matrix $A\in\mathrm{O}\left(2,\,\mathbb{E}\right)$,
the \emph{orthogonal generator} $o_{\mathbb{E}}\left[A\right]$ is
defined to be
\begin{align*}
o_{\mathbb{E}}\left[A\right]\, & :=\,\left[\begin{array}{cc}
1 & \overrightarrow{0}^{T}\\
\overrightarrow{0} & A
\end{array}\right],
\end{align*}
where $\overrightarrow{0}$ is the two-dimensional zero vector.

\noindent \end{definition}
\begin{lemma}

For a Euclidean field $\mathbb{E}$, the group $\mathrm{PU}\left(\mathbb{E}\right)$
acts transitively on the vertices of the Bruhat-Tits building $X_{\mathbb{E}\left(t\right)}$.

\noindent \end{lemma}
\begin{proof}

\noindent The elements in the link of the identity lattice $\left[I\right]$,
$\mathrm{Lk}_{\left[I\right]}$, correspond to the elements in the
projective plane $\mathbb{E}\mathbf{P}^{2}$. More concretely, the
elements of type 1 correspond to the points in $\mathbb{E}\mathbf{P}^{2}$,
and the elements of type -1 correspond to the the lines in $\mathbb{E}\mathbf{P}^{2}$.
Consider lattices made of the elementary generators. For $r>0$, a
computation yields
\begin{align*}
\left[k_{\mathbb{E}}\left[r\right]\right] & =\left[\left[\begin{array}{ccc}
\frac{r^{2}-t}{\sqrt{1+r^{4}}} & \frac{-rt}{\sqrt{1+r^{4}}} & 0\\
\frac{-r\left(t^{-1}+t\right)}{\sqrt{1+r^{4}}} & \frac{1-r^{2}t}{\sqrt{1+r^{4}}} & 0\\
0 & 0 & 1
\end{array}\right]\right]\\
 & =\left[\left[\begin{array}{ccc}
r^{2}-t & -rt & 0\\
-r\left(t^{-1}+t\right) & 1-r^{2}t & 0\\
0 & 0 & 1
\end{array}\right]\right]\\
 & =\left[\left[\begin{array}{ccc}
0 & -rt & 0\\
-\frac{1+r^{4}}{r} & 1-r^{2}t & 0\\
0 & 0 & 1
\end{array}\right]\right]\\
 & =\left[\left[\begin{array}{ccc}
t & 0 & 0\\
rt & 1 & 0\\
0 & 0 & 1
\end{array}\right]\right],
\end{align*}
where for the second and the last equalities, we used the fact $1+r^{4}>0$.

Therefore, for each $r>0$ and each special orthogonal matrix $SO\in\mathrm{O}\left(2,\,\mathbb{E}\right)$,
the lattice
\begin{align*}
\left[o_{\mathbb{E}}\left[SO\right]k_{\mathbb{E}}\left[r\right]o_{\mathbb{E}}\left[SO\right]^{-1}\right]
\end{align*}
represents a unique point in $\mathbb{E}\mathbf{P}^{2}$. The other
points in $\mathbb{E}\mathbf{P}^{2}$ correspond to $k_{\mathbb{E}}\left[0\right]$
and
\begin{align*}
\left[o_{\mathbb{E}}\left[SO\right]k_{\mathbb{E}}\left[\infty\right]o_{\mathbb{E}}\left[SO\right]^{-1}\right].
\end{align*}

Note that there is some overlap among $\left[o_{\mathbb{E}}\left[SO\right]k_{\mathbb{E}}\left[\infty\right]o_{\mathbb{E}}\left[SO\right]^{-1}\right]$.
Considering the inverses of type 1 lattices, each lattice in $\mathrm{Lk}_{\left[I\right]}$
is represented by acting some element in $\mathrm{PU}\left(\mathbb{E}\right)$
on the identity lattice $\left[I\right]$, which concludes the proof.
$\qedhere$

\noindent \end{proof}

The stabilizer of the identity lattice $\left[I\right]$ is the subgroup
of $\mathrm{PU}\left(\mathbb{E}\right)$ generated by the matrices
in which every entry is in $\mathbb{E}$. Therefore, the stabilizer
is exactly the subgroup generated by the orthogonal generators $o_{\mathbb{E}}\left[A\right]$,
which is isomorphic to $\mathrm{O}\left(2,\,\mathbb{E}\right)$. Since
the building is a CAT(0) space, it is simply connected, and we can
write the presentation of $\mathrm{PU}\left(\mathbb{E}\right)$ only
in terms of the relations of $\mathrm{O}\left(2,\,\mathbb{E}\right)$
and those given within the link of the identity lattice. The relations
involving elementary generators are given by
\begin{align}
k_{\mathbb{E}}\left[\infty\right]k_{\mathbb{E}}\left[0\right] & =o_{\mathbb{E}}\left[\begin{array}{cc}
0 & -1\\
1 & 0
\end{array}\right]k_{\mathbb{E}}\left[\infty\right]^{-1}o_{\mathbb{E}}\left[\begin{array}{cc}
0 & 1\\
-1 & 0
\end{array}\right],\\
k_{\mathbb{E}}\left[r\right]o_{\mathbb{E}}\left[\begin{array}{cc}
-1 & 0\\
0 & -1
\end{array}\right]k_{\mathbb{E}}\left[r^{-1}\right] & =o_{\mathbb{E}}\left[\begin{array}{cc}
0 & 1\\
-1 & 0
\end{array}\right]k_{\mathbb{E}}\left[\infty\right]^{-1}o_{\mathbb{E}}\left[\begin{array}{cc}
0 & 1\\
-1 & 0
\end{array}\right],
\end{align}
where $r>0$;
\begin{align}
\begin{array}{cc}
 & k_{\mathbb{E}}\left[\sqrt{r_{1}}\right]o_{\mathbb{E}}\left[\begin{array}{cc}
-\sqrt{\frac{1}{r_{1}r_{2}}} & -\sqrt{1-\frac{1}{r_{1}r_{2}}}\\
\sqrt{1-\frac{1}{r_{1}r_{2}}} & -\sqrt{\frac{1}{r_{1}r_{2}}}
\end{array}\right]k_{\mathbb{E}}\left[\sqrt{r_{2}}\right]\\
= & o_{\mathbb{E}}\left[\begin{array}{cc}
-\sqrt{\frac{r_{2}-r_{1}^{-1}}{r_{1}+r_{2}}} & \sqrt{\frac{r_{1}+r_{1}^{-1}}{r_{1}+r_{2}}}\\
-\sqrt{\frac{r_{1}+r_{1}^{-1}}{r_{1}+r_{2}}} & -\sqrt{\frac{r_{2}-r_{1}^{-1}}{r_{1}+r_{2}}}
\end{array}\right]k_{\mathbb{E}}\left[\sqrt{\frac{r_{1}+r_{2}}{r_{1}r_{2}-1}}\right]^{-1}o_{\mathbb{E}}\left[\begin{array}{cc}
-\sqrt{\frac{r_{1}-r_{2}^{-1}}{r_{1}+r_{2}}} & \sqrt{\frac{r_{2}+r_{2}^{-1}}{r_{1}+r_{2}}}\\
-\sqrt{\frac{r_{2}+r_{2}^{-1}}{r_{1}+r_{2}}} & -\sqrt{\frac{r_{1}-r_{2}^{-1}}{r_{1}+r_{2}}}
\end{array}\right],
\end{array}
\end{align}
where $r_{1},\,r_{2}>0$ and $r_{1}r_{2}>1$;
\begin{align*}
\left[k_{\mathbb{E}}\left[0\right],\,o_{\mathbb{E}}\left[A\right]\right] & =1,
\end{align*}
for each $A\in\mathrm{O}\left(2,\,\mathbb{E}\right)$;
\begin{align*}
\left[k_{\mathbb{E}}\left[\infty\right],\,o_{\mathbb{E}}\left[\begin{array}{cc}
-1 & 0\\
0 & -1
\end{array}\right]\right] & =1;
\end{align*}
and
\begin{align*}
\left[k_{\mathbb{E}}\left[r\right],\,o_{\mathbb{E}}\left[\begin{array}{cc}
1 & 0\\
0 & -1
\end{array}\right]\right] & =1,
\end{align*}
for $r\ge0$ or $r=\infty$.

To proceed further, we ask whether there is a subgroup of $\mathrm{PU}\left(\mathbb{E}\right)$
that acts transitively on the vertices of the building with small
stabilizers. At this stage, it is challenging to find a suitable answer
for the fields $\mathbb{R}$ or $\mathbb{E}$. Let us extend the field
once more. For a Euclidean field $\mathbb{E}$, we consider the field
$\mathbb{E}\left(i\right)$, where $i^{2}=-1$. An important distinction
arises here: the bar involution $\overline{\left(\right)}$ we defined
in Section 1 for the Laurent polynomial ring $\mathbb{E}\left[t,\,t^{-1}\right]$
must now be extended compatibly with $i$. We extend the involution
in $\mathbb{E}\left(i\right)\left[t,\,t^{-1}\right]$ such that $\overline{i}=-i$.
Then, we construct the building $X_{\mathbb{E}\left(i\right)\left(t\right)}$
in the same manner and define the projective unitary group $\mathrm{PU}\left(\mathbb{E}\left(i\right)\right)$.

\noindent \begin{definition}

For a Euclidean field $\mathbb{E}$, for each element $r\in\mathbb{E}$
such that $0<r$, the \emph{unipotent generator} $\mathbf{k}_{\mathbb{E}\left(i\right)}\left[r\right]$
in $\mathrm{PU}\left(\mathbb{E}\left(i\right)\right)$ is defined
to be
\begin{align*}
\mathbf{k}_{\mathbb{E}\left(i\right)}\left[r\right]\, & :=\,k_{\mathbb{E}}\left[r\right]\left[\begin{array}{ccc}
1 & 0 & 0\\
0 & \frac{r^{2}+i}{1+ir^{2}} & 0\\
0 & 0 & \frac{r^{2}+i}{\sqrt{1+r^{4}}}
\end{array}\right].
\end{align*}

Moreover, we add two unipotent generators $\mathbf{k}_{\mathbb{E}\left(i\right)}\left[0\right]$\emph{
}and $\mathbf{k}_{\mathbb{E}\left(i\right)}\left[\infty\right]$ as
follows.
\begin{align*}
\mathbf{k}_{\mathbb{E}\left(i\right)}\left[0\right] & :=\left[\begin{array}{ccc}
it & 0 & 0\\
0 & 1 & 0\\
0 & 0 & 1
\end{array}\right],\:\mathbf{k}_{\mathbb{E}\left(i\right)}\left[\infty\right]:=\left[\begin{array}{ccc}
1 & 0 & 0\\
0 & it & 0\\
0 & 0 & 1
\end{array}\right].
\end{align*}

For each (usual) unitary matrix $A\in\mathrm{U}\left(2,\,\mathbb{E}\left(i\right)\right)$,
the \emph{unitary generator} $u_{\mathbb{E}\left(i\right)}\left[A\right]$
are defined to be
\begin{align*}
u_{\mathbb{E}\left(i\right)}\left[A\right]\, & :=\,\left[\begin{array}{cc}
1 & \overrightarrow{0}^{T}\\
\overrightarrow{0} & A
\end{array}\right],
\end{align*}
where $\overrightarrow{0}$ is the two-dimensional zero vector.

Finally, the \emph{unipotent kernel} $\mathcal{U}_{\mathbb{E}\left(i\right)}\subset\mathrm{PU}\left(\mathbb{E}\left(i\right)\right)$
is the set of elements $B\in\mathrm{PU}\left(\mathbb{E}\left(i\right)\right)$
such that the evaluation at $t=-i$ is unipotent, or
\begin{align*}
B|_{t=-i} & =\left[\begin{array}{ccc}
1 & * & *\\
0 & 1 & *\\
0 & 0 & 1
\end{array}\right].
\end{align*}
\end{definition}
\begin{theorem}

For a Euclidean field $\mathbb{E}$, the unipotent kernel $\mathcal{U}_{\mathbb{E}\left(i\right)}$
is generated by the unipotent generators and their conjugates by unitary
generators $u_{\mathbb{E}}\left[A\right]$ such that $A\in\mathrm{SU}\left(2,\,\mathbb{E}\left(i\right)\right)$.
The group $\mathcal{U}_{\mathbb{E}\left(i\right)}$ acts on the building
$X_{\mathbb{E}\left(i\right)\left(t\right)}$ transitively on the
vertices with the trivial stabilizers, and is a normal subgroup of
$\mathrm{PU}\left(\mathbb{E}\left(i\right)\right)$. The quotient
map is split, and makes $\mathrm{PU}\left(\mathbb{E}\left(i\right)\right)$
a semidirect product:
\begin{align*}
\mathrm{PU}\left(\mathbb{E}\left(i\right)\right) & =\mathcal{U}_{\mathbb{E}\left(i\right)}\rtimes\mathrm{U}\left(2,\,\mathbb{E}\left(i\right)\right),
\end{align*}
where the unitary group $\mathrm{U}\left(2,\,\mathbb{E}\left(i\right)\right)$
is identified with $u_{\mathbb{E}\left(i\right)}\left[\mathrm{U}\left(2,\,\mathbb{E}\left(i\right)\right)\right]$.

\noindent \end{theorem}
\begin{proof}

When $r\ge0$ or $r=\infty$, the unipotent generators satisfy
\begin{align*}
\mathbf{k}_{\mathbb{E}\left(i\right)}\left[r\right]|_{t=-i}\, & =\,\left[\begin{array}{ccc}
1 & \frac{r}{r^{2}-i} & 0\\
0 & 1 & 0\\
0 & 0 & 1
\end{array}\right],\,\mathbf{k}_{\mathbb{E}\left(i\right)}\left[\infty\right]|_{t=-i}\,=\,\left[\begin{array}{ccc}
1 & 0 & 0\\
0 & 1 & 0\\
0 & 0 & 1
\end{array}\right].
\end{align*}

Thus, the unipotent kernel $\mathcal{U}_{\mathbb{E}\left(i\right)}$
includes $\mathbf{k}_{\mathbb{E}\left(i\right)}\left[r\right]$ and
its every conjugate by unitary generators $u_{\mathbb{E}}\left[A\right]$,
where $A\in\mathrm{SU}\left(2,\,\mathbb{E}\left(i\right)\right)$.
As demonstrated in the proof of Lemma 4.3, these elements represent
the type 1 lattices in $\mathrm{Lk}_{\left[I\right]}$. Consequently,
$\mathcal{U}_{\mathbb{E}\left(i\right)}$ acts transitively on the
vertices of $X_{\mathbb{E}\left(i\right)\left(t\right)}$. Furthermore,
the stabilizer of the identity lattice $\left[I\right]$ is trivial,
as the identity matrix is the only unipotent element in $\mathrm{U}\left(2,\,\mathbb{E}\left(i\right)\right)$.

Since $\mathcal{U}_{\mathbb{E}\left(i\right)}$ acts simply transitively
on the vertices of $X_{\mathbb{E}\left(i\right)\left(t\right)}$,
the unipotent generators $\mathbf{k}_{\mathbb{E}\left(i\right)}\left[r\right]$
and their congugates generate the entire unipotent kernel $\mathcal{U}_{\mathbb{E}\left(i\right)}$.
The evaluated image of $\mathrm{PU}\left(\mathbb{E}\left(i\right)\right)$
at $t=-i$ forms a subgroup of the affine group, where the first column
normalizes to $\left[\begin{array}{ccc}
1 & 0 & 0\end{array}\right]^{T}$ by setting the $\left(1,\,1\right)$-entry to 1. At $t=-i$, the
evaluated images of elements of $\mathcal{U}_{\mathbb{E}\left(i\right)}$
have zeros in the $\left(2,\,3\right)$- and $\left(3,\,2\right)$-entries,
indicating that $\mathcal{U}_{\mathbb{E}\left(i\right)}$ is normal
in $\mathrm{PU}\left(\mathbb{E}\left(i\right)\right)$.

Moreover, since $\mathcal{U}_{\mathbb{E}\left(i\right)}$ and $u_{\mathbb{E}\left(i\right)}\left[\mathrm{U}\left(2,\,\mathbb{E}\left(i\right)\right)\right]$
have trivial intersection and together generate $\mathrm{PU}\left(\mathbb{E}\left(i\right)\right)$,
the entire group $\mathrm{PU}\left(\mathbb{E}\left(i\right)\right)$
is the semidirect product of these two subgroups. $\qedhere$

\noindent \end{proof}

Theorem 4.5 means the unipotent kernel $\mathcal{U}_{\mathbb{E}\left(i\right)}$
acts simply transitively on the vertices of $X_{\mathbb{E}\left(i\right)\left(t\right)}$.
Moreover, the action of $\mathcal{U}_{\mathbb{E}\left(i\right)}$
is type-rotating, by observing the determinants of the generators.
This type of groups, or $\widetilde{A_{2}}$ groups, has been studied
extensively since the classical works by Cartwright, Mantero, Steger
and Zappa in \citep{MR1232965,MR1232966}. According to \citep[Theorem 3.1]{MR1232965},
the structure of the group $\mathcal{U}_{\mathbb{E}\left(i\right)}$
is determined by the geometry of the projective plane $\mathbb{E}\left(i\right)\mathbf{P}^{2}$,
through the triangle presentation. Concretely, a practical triangle
presentation for $\mathcal{U}_{\mathbb{E}\left(i\right)}$ is constructed
by rewriting equations (18), (19) and (20) for unipotent generators
and taking conjugations by unitary generators. The palindromic polynomial
$t^{-1}+t$ in the diagonal of the matrix $D$ plays a crucial role
in determining the triangle presentation. Modifying it to $t^{-1}+1+t$,
$t^{-1}-1+t$, etc., results in unipotent kernels that exhibit different
types of triangle presentations, while acting on the same building.

\noindent \begin{prooftheorem14}

\noindent All the arguments presented so far apply to general Euclidean
fields. Therefore, it suffices to show that for any given infinite
cardinality $\kappa$, there exists a Euclidean field $\mathbb{E}$
of cardinality $\kappa$. Since the theory of ordered fields is (countably)
first-order, the theory of Euclidean fields is also first-order. There
are at least two models of Euclidean fields, the field of real constructible
numbers, and the field of real numbers $\mathbb{R}$. By applying
the upward Loewenheim\textendash Skolem theorem \citep{MR1924282},
we conclude the proof. $\qed$

\noindent \end{prooftheorem14}

Finally, we now turn our attention again to the Burau representation
of $B_{4}$. Recall that we defined two matrices $S'$ and $T'$,
given by (10) as
\begin{align*}
S' & =\left(\begin{array}{ccc}
1 & 0 & 0\\
0 & -t & 0\\
0 & 0 & -t^{-1}
\end{array}\right),\;T':=\left(\begin{array}{ccc}
-t^{-1}+1-t & 0 & t\left(1-t\right)\\
\left(t^{-1}-1\right)\left(t^{-1}+t\right) & 0 & -1+t-t^{2}\\
0 & -t^{-1} & 0
\end{array}\right).
\end{align*}

Birman's theorem guarantees that the representation $\beta_{4,\,r}$
is faithful if the image of $\left\langle S',\,T'\right\rangle $
generates the free group of rank 2. Evaluating at $t=-i$, we have
\begin{align*}
S' & |_{t=-i}=\left(\begin{array}{ccc}
1 & 0 & 0\\
0 & i & 0\\
0 & 0 & -i
\end{array}\right),\;T'|_{t=-i}:=\left(\begin{array}{ccc}
1 & 0 & 1-i\\
0 & 0 & -i\\
0 & -i & 0
\end{array}\right).
\end{align*}

Regarding these as elements of the affine group, an index 8 normal
subgroup $\mathcal{G}$ of $\left\langle S',\,T'\right\rangle $ is
contained in the unipotent kernel. By using the the Reidemeister-Schreier
process, we list the 9 generators of $\mathcal{G}$ as follows:
\begin{align*}
a_{1} & :=S'^{2}T'^{-2},\,a_{2}:=S'T'^{3}S'^{-1}T'^{-1},\,a_{3}:=S'T'S'T'^{-1},\\
a_{4} & :=T'^{2}S'T'S'^{-1}T'^{-1},\,a_{5}:=T'^{4},\,a_{6}:=T'^{3}S'T'^{-1}S'^{-1},\\
a_{7} & :=T'S'T'^{-1}S',\,a_{8}:=T'S'T'S'^{-1},\,a_{9}:=T'S'^{2}T'^{-3}.
\end{align*}

The choice of the base field $\mathbb{E}\left(i\right)$ is not important
now; we can simply take it as $\mathbb{C}$. These elements are generated
by a finite set of generators of $\mathcal{U}_{\mathbb{C}}$. We label
a set of generators as follows:
\begin{align*}
d_{1} & :=\mathbf{k}_{\mathbb{C}}\left[\infty\right],\,d_{2}:=u_{\mathbb{C}}\left[\begin{array}{cc}
0 & 1\\
-1 & 0
\end{array}\right]\mathbf{k}_{\mathbb{C}}\left[\infty\right]u_{\mathbb{C}}\left[\begin{array}{cc}
0 & -1\\
1 & 0
\end{array}\right],\\
g_{1} & :=\mathbf{k}_{\mathbb{C}}\left[1\right],\,g_{2}:=u_{\mathbb{C}}\left[\begin{array}{cc}
0 & 1\\
-1 & 0
\end{array}\right]\mathbf{k}_{\mathbb{C}}\left[1\right]u_{\mathbb{C}}\left[\begin{array}{cc}
0 & -1\\
1 & 0
\end{array}\right],\\
g_{3} & :=u_{\mathbb{C}}\left[\begin{array}{cc}
i & 0\\
0 & -i
\end{array}\right]\mathbf{k}_{\mathbb{C}}\left[1\right]u_{\mathbb{C}}\left[\begin{array}{cc}
-i & 0\\
0 & i
\end{array}\right],\,g_{4}:=u_{\mathbb{C}}\left[\begin{array}{cc}
0 & i\\
i & 0
\end{array}\right]\mathbf{k}_{\mathbb{C}}\left[1\right]u_{\mathbb{C}}\left[\begin{array}{cc}
0 & -i\\
-i & 0
\end{array}\right].
\end{align*}

Then, the subgroup $\mathcal{H}:=\left\langle d_{1},\,d_{2},\,g_{1},\,g_{2},\,g_{3},\,g_{4}\right\rangle $
of $\mathcal{U}_{\mathbb{C}}$ contains $\mathcal{G}$. Concretely,
we have
\begin{align}
\begin{cases}
 & a_{1}=d_{1}^{3}d_{2}^{-3}g_{4}^{2}g_{1}^{-2},\,a_{2}=d_{1}d_{2}^{-1}g_{3}^{2}d_{1}g_{2}^{2}d_{2}^{-1}g_{3}^{-2}d_{1}d_{2}^{-1}g_{1}^{-2},\,a_{3}=d_{1}g_{3}^{2}d_{1}^{-1}g_{1}^{-2},\\
 & a_{4}=g_{1}^{2}g_{4}^{-2}d_{2}^{-3}g_{3}^{-2}d_{1}g_{1}^{-2},\,a_{5}=g_{1}^{2}d_{1}^{-1}g_{4}^{-2}d_{2}^{-1}g_{1}^{-2}d_{1}g_{4}^{2}d_{2},\\
 & a_{6}=g_{1}^{2}d_{1}^{-1}g_{4}^{-2}d_{2}g_{1}^{-2}d_{1}^{-1}g_{3}^{-2}d_{1}^{-1},\,a_{7}=d_{2}^{2}g_{1}^{2}d_{1}^{-1}g_{3}^{2}d_{1},\\
 & a_{8}=d_{2}g_{1}^{2}d_{1}^{-3}g_{2}^{-2}d_{2}^{2},\,a_{9}=d_{2}^{3}g_{1}^{2}d_{1}^{-2}g_{1}^{2}d_{1}g_{4}^{2}g_{1}^{-2}.
\end{cases}
\end{align}
\begin{definition}

For a subgroup $G$ of $\mathcal{U}_{\mathbb{E}\left(i\right)}$,
define the \emph{spanning subcomplex} $\Delta_{\mathbb{E}\left(i\right)}\left(G\right)$
to be the simplicial subcomplex of $X_{\mathbb{E}\left(i\right)\left(t\right)}$
such that (1) the set of the vertices of $\Delta_{\mathbb{E}\left(i\right)}\left(G\right)$
is $\left[G\right]$, (2) the complex has an edge $\left(v_{1},\,v_{2}\right)$
if and only if $v_{1},\,v_{2}\in\left[G\right]$, and (3) the complex
has a 2-cell if and only if the complex has the three sides of the
2-cell.

\noindent \end{definition}

We establish that the spanning subcomplex $\Delta_{\mathbb{E}\left(i\right)}\left(\mathcal{H}\right)$
is \emph{locally finite}. In other words, the link of a vertex in
$\Delta_{\mathbb{E}\left(i\right)}\left(\mathcal{H}\right)$ contains
only finitely many vertices. Notably, the proof of this finiteness
relies on a complete (infinite) classification of the generators of
$\mathcal{U}_{\mathbb{E}\left(i\right)}$.

\noindent \begin{lemma}

The spanning subcomplex $\Delta_{\mathbb{E}\left(i\right)}\left(\mathcal{H}\right)$
contains exactly 11 vertices of type 1 in $\mathrm{Lk}_{\left[I\right]}$,
as follows.
\begin{align*}
 & \left[d_{1}\right],\,\left[d_{2}\right],\,\left[d_{1}^{-1}d_{2}^{-1}\right],\,\left[g_{1}\right],\,\left[g_{2}\right],\,\left[g_{3}\right],\,\left[g_{4}\right],\\
 & \left[d_{2}^{-1}g_{1}^{-1}\right],\,\left[d_{2}^{-1}g_{3}^{-1}\right],\,\left[d_{1}^{-1}g_{2}^{-1}\right],\,\left[d_{1}^{-1}g_{4}^{-1}\right].
\end{align*}
\end{lemma}
\begin{proof}

It is evident that the given 11 vertices are of type 1 and belong
to $\mathrm{Lk}_{\left[I\right]}$. To complete the argument, it remains
to show that there are no additional type 1 vertices.

First, for a special unitary matrix $A\in\mathrm{SU}\left(2,\,\mathbb{E}\left(i\right)\right)$
of the form
\begin{align*}
A & =\left(\begin{array}{cc}
a_{1} & a_{2}\\
-\overline{a_{2}} & \overline{a_{1}}
\end{array}\right),
\end{align*}
consider the matrix $u_{\mathbb{E}\left(i\right)}\left[A\right]\mathbf{k}_{\mathbb{E}\left(i\right)}\left[\infty\right]u_{\mathbb{E}\left(i\right)}\left[A\right]^{-1}$.
When evaluated at $t=i$, the $\left(2,\,3\right)$-entry is $2a_{1}a_{2}$,
and the $\left(3,\,2\right)$-entry is $2\overline{a_{1}a_{2}}$.
Since the generators of $\mathcal{H}$ have 0 in these positions when
evaluated at $t=i$, the only possibilities are $a_{1}=0$ or $a_{2}=0$,
corresponding to the generators $d_{2}$ or $d_{1}$, respectively.

Next, consider the case $u_{\mathbb{E}\left(i\right)}\left[A\right]\mathbf{k}_{\mathbb{E}\left(i\right)}\left[r\right]u_{\mathbb{E}\left(i\right)}\left[A\right]^{-1}$
for a nonnegative element $r\in\mathbb{E}$ and a matrix $A\in\mathrm{SU}\left(2,\,\mathbb{E}\left(i\right)\right)$.
Define the map $\phi:\mathcal{U}_{\mathbb{E}\left(i\right)}\to\mathbb{E}\left(i\right)^{2}$
as the composition of the evaluation at $t=-i$ with a map defined
by
\begin{align*}
\left[\begin{array}{ccc}
1 & b_{12} & b_{13}\\
0 & 1 & 0\\
0 & 0 & 1
\end{array}\right] & \mapsto\left(b_{12},\,b_{13}\right).
\end{align*}

In general, we have
\begin{align}
\phi\left(u_{\mathbb{E}}\left[\begin{array}{cc}
a_{1} & a_{2}\\
-\overline{a_{2}} & \overline{a_{1}}
\end{array}\right]\mathbf{k}_{\mathbb{E}}\left[r\right]u_{\mathbb{E}}\left[\begin{array}{cc}
a_{1} & a_{2}\\
-\overline{a_{2}} & \overline{a_{1}}
\end{array}\right]^{-1}\right) & =\left(\frac{\overline{a_{1}}r\left(r^{2}+i\right)}{r^{4}+1},\,\frac{-a_{2}r\left(r^{2}+i\right)}{r^{4}+1}\right),
\end{align}

\noindent where $a_{1}\overline{a_{1}}+a_{2}\overline{a_{2}}=1$.
The two generators $d_{1}$ and $d_{2}$ have the trivial image $\left(0,\,0\right)$
under $\phi$. The images of the others are given by
\begin{align*}
\phi\left(g_{1}\right) & =\left(\frac{1}{2}+\frac{i}{2},\,0\right),\,\phi\left(g_{2}\right)=\left(0,\,-\frac{1}{2}-\frac{i}{2}\right),\\
\phi\left(g_{3}\right) & =\left(\frac{1}{2}-\frac{i}{2},\,0\right),\,\phi\left(g_{4}\right)=\left(0,\,\frac{1}{2}-\frac{i}{2}\right).
\end{align*}

It suffices to determine which numbers $a_{1},\,a_{2},\,r$ make the
right-hand side of (22) belong to $\left(\mathbb{Z}\left[i\right]+\left(\frac{1+i}{2}\right)\right)^{2}$.
By considering the \emph{norm map}
\begin{align*}
\left(a,\,b\right) & \mapsto a\overline{a}+b\overline{b}\in\mathbb{Z}+\frac{1}{2},
\end{align*}
we obtain
\begin{align*}
\frac{r^{2}}{r^{4}+1}=0\;\;\mathrm{or}\;\;\frac{1}{2} & \le\frac{r^{2}}{r^{4}+1},
\end{align*}
where the first condition implies $r=0$, which corresponds to $\left[d_{1}^{-1}d_{2}^{-1}\right]$,
and the second condition implies $r^{2}=1$. Assuming the latter,
we conclude $r=1$, as $r$ is assumed to be nonnegative. Substituting
$r=1$, the right-hand side of (22) simplifies to
\begin{align*}
\left(\frac{\overline{a_{1}}\left(1+i\right)}{2},\,\frac{-a_{2}\left(1+i\right)}{2}\right).
\end{align*}

Computing the norms $\frac{a_{1}\overline{a_{1}}}{2}$ and $\frac{a_{2}\overline{a_{2}}}{2}$
of the two entries, which are either 0 or not smaller than $\frac{1}{2}$,
we observe that $a_{1}$ and $a_{2}$ are either 0 or have norm 1.
Using the assumption $a_{1}\overline{a_{1}}+a_{2}\overline{a_{2}}=1$,
the only possible cases are $a_{j}=\pm1,\,\pm i$ and $a_{3-j}=0$
for $j=1,\,2$. $\qedhere$

\noindent \end{proof}
\begin{prooftheorem15}

\noindent The goal of the proof is to prove that the group $\mathcal{G}$
is free of rank 9 as a subgroup of $\mathcal{H}$. Consider the spanning
subcomplex $\Delta_{\mathbb{C}}\left(\mathcal{H}\right)$. Suppose
this subcomplex is simply connected, which is the premise of Theorem
1.5. Then, the relations of the group $\mathcal{H}$ are determined
in the link of the identity lattice, $\mathrm{Lk}_{\left[I\right]}$.
By Lemma 4.7, there are exactly 11 type 1 vertices in the link, meaning
only finitely many relations are possible. Using the relations (18),
(19) and (20), or by computing them one by one, we find that there
are 5 relations after continually using the Tietze transformation.
Then, the group structure of $\mathcal{H}$ is described as
\begin{align*}
\left(\left\langle d_{1},\,g_{1},\,g_{3}\right\rangle \times\left\langle d_{2}\right\rangle \right)*_{\left\langle d_{1},\,d_{2}\right\rangle }\left(\left\langle d_{2},\,g_{2},\,g_{4}\right\rangle \times\left\langle d_{1}\right\rangle \right),
\end{align*}
where both $\left\langle d_{1},\,g_{1},\,g_{3}\right\rangle $ and
$\left\langle d_{2},\,g_{2},\,g_{4}\right\rangle $ are free of rank
3. The symbol $*_{\left\langle d_{1},\,d_{2}\right\rangle }$ means
the group is a free product amalgamated at the group $\left\langle d_{1}\right\rangle \times\left\langle d_{2}\right\rangle \cong\mathbb{Z}^{2}$.
Therefore, the group $\mathcal{H}$ is an example of a right-angled
Artin group whose defining graph is a tree, which is a well-understood
class of groups (see \citep{MR2322545} for a general introduction).

There are many possible ways to prove the freeness of $\mathcal{G}$,
but for our context, we define a quotient map $\mathcal{\mathcal{F}}:\mathcal{H}\to\mathbb{Z}$,
adding relations $d_{1}=g_{1}^{-2}=d_{2}$, $g_{1}=g_{2}=g_{3}=g_{4}$.
It can then be shown inductively that $\ker\mathcal{F}$ is free of
rank 9 generated by
\begin{align*}
l_{1} & :=d_{2}g_{1}^{2},\,l_{2}:=d_{1}g_{1}^{2},\,l_{3}:=g_{1}d_{1}g_{1},\\
l_{4} & :=g_{2}^{-1}g_{1},\,l_{5}:=g_{1}g_{2}^{-1},\,l_{6}:=g_{3}^{-1}g_{1},\\
l_{7} & :=g_{1}g_{3}^{-1},\,l_{8}:=g_{4}^{-1}g_{1},\,l_{9}:=g_{1}g_{4}^{-1}.
\end{align*}

Using (21), we verify $\mathcal{F}\left(a_{j}\right)$ is trivial
for each $j$ such that $1\le j\le9$. By employing the free generators
$l_{j}$, we have
\begin{align*}
\begin{cases}
 & a_{1}=\left(l_{2}l_{1}^{-1}\right)^{3}l_{9}l_{3}l_{8}^{-1}l_{3}^{-1},\,a_{2}=l_{2}l_{1}^{-1}l_{7}^{-1}l_{1}l_{6}^{-1}l_{2}l_{1}^{-1}l_{5}^{-1}l_{3}l_{4}^{-1}l_{1}l_{2}l_{6}l_{1}^{-1}l_{7}l_{2}l_{1}^{-2},\\
 & a_{3}=l_{2}l_{1}^{-1}l_{7}^{-1}l_{1}l_{6}^{-1}l_{1}l_{2}^{-1}l_{1}^{-1},\,a_{4}=l_{1}l_{2}l_{1}^{-1}l_{8}l_{3}^{-1}l_{9}l_{6}l_{1}^{-1}l_{7}l_{2}l_{1}^{-2},\\
 & a_{5}=l_{1}l_{2}l_{1}^{-1}l_{2}l_{1}^{-1}l_{8}l_{3}^{-1}l_{1}l_{3}^{-2}l_{1}^{-1}l_{9}l_{1}l_{2}^{-1}l_{1}^{-1}l_{2}l_{1}^{-1}l_{9}^{-1}l_{3}l_{8}^{-1}l_{1}l_{2},\\
 & a_{6}=l_{1}l_{2}l_{1}^{-1}l_{8}l_{3}^{-1}l_{1}l_{3}^{-2}l_{1}^{-1}l_{9}l_{1}l_{2}^{-2}l_{6}l_{1}^{-1}l_{7}l_{1}l_{2}^{-1},\,a_{7}=l_{1}^{2}l_{2}^{-1}l_{7}^{-1}l_{1}l_{6},\\
 & a_{8}=l_{1}l_{4}l_{3}^{-1}l_{5}l_{1}l_{2}l_{1}l_{2}^{-1},\,a_{9}=l_{1}^{2}l_{2}^{-1}l_{1}l_{2}^{-1}l_{1}l_{2}l_{1}^{-1}l_{9}^{-1}l_{3}l_{8}^{-1}l_{1}l_{2}^{-1}l_{1}^{-1}.
\end{cases}
\end{align*}

Since $\mathcal{G}$ is a finitely generated subgroup of a finitely
generated free group, determining its free rank is a combinatorial
problem. This can be solved algorithmically using Stallings' folding
\citep{MR695906,MR1882114}. Applying this algorithm yields the desired
result that the rank is 9. $\qed$

\noindent \end{prooftheorem15}
\begin{remark}

The map $\mathcal{F}:\mathcal{H}\to\mathbb{Z}$, defined in the proof
of Theorem 1.5, can be extended to the entire unipotent kernel $\mathcal{U}_{\mathbb{C}}$,
without assuming the simple connectivity of $\Delta_{\mathbb{C}}\left(\mathcal{H}\right)$.
This extension is achieved by defining the quotient map $\widetilde{\mathcal{F}}:\mathcal{U}_{\mathbb{C}}\to\mathcal{U}_{\mathbb{C}}/\ker\widetilde{\mathcal{F}}$
by equating a generator with its conjugates by all unitary generators.
The image of $\widetilde{\mathcal{F}}$ forms an abelian group. Moreover,
the quotient of $\mathcal{U}_{\mathbb{C}}/\ker\widetilde{\mathcal{F}}$
by $\widetilde{\mathcal{F}}\left(\left\langle \mathbf{k}_{\mathbb{C}}\left[\infty\right]\right\rangle \right)$
forms a divisible group $\mathbf{D}$, in which the image of $\left\langle \mathbf{k}_{\mathbb{C}}\left[1\right]\right\rangle $
constitutes a torsion subgroup of order 2. Elements of order $2^{n}$
are not the only torsion elements in $\mathbf{D}$; for instance,
the image of $\mathbf{k}_{\mathbb{C}}\left[3^{\frac{1}{4}}\right]$
is a torsion element of order 3 in $\mathbf{D}$. The map $\widetilde{\mathcal{F}}$
may play a role for future research directions.

\noindent \end{remark}
\begin{remark}

We make a brief remark on the computational aspects. The link of the
identity lattice within our subcomplex $\Delta_{\mathbb{C}}\left(\mathcal{H}\right)$
is not a spherical building, which implies that $\Delta_{\mathbb{C}}\left(\mathcal{H}\right)$
is not a subbuilding of $X_{\mathbb{C}\left(t\right)}$ (or $X_{\mathbb{Q}\left(i\right)\left(t\right)}$).
While one might consider the minimal subbuilding $\Delta'$ containing
$\Delta_{\mathbb{C}}\left(\mathcal{H}\right)$, computations in $\Delta'$
are generally quite challenging. This difficulty appears to stem from
the properties of the field of real constructible numbers, or at least
from the properties of Euclidean fields. Notably, a recent result
\citep{MR4772285} has established that the field of real constructible
numbers is undecidable.

\noindent \end{remark}

\begin{spacing}{0.9}
\bibliographystyle{amsplain}
\phantomsection\addcontentsline{toc}{section}{\refname}\bibliography{bibgen}

\end{spacing}

$ $

{\small{}Donsung Lee; \href{mailto:disturin@snu.ac.kr}{disturin@snu.ac.kr}}{\small\par}

{\small{}Department of Mathematical Sciences and Research Institute
of Mathematics,}{\small\par}

{\small{}Seoul National University, Gwanak-ro 1, Gwankak-gu, Seoul,
South Korea 08826}{\small\par}

\clearpage{}

\pagebreak{}

\pagenumbering{arabic}

\renewcommand{\thefootnote}{A\arabic{footnote}}
\renewcommand{\thepage}{A\arabic{page}}
\renewcommand{\thetable}{A\arabic{table}}
\renewcommand{\thefigure}{A\arabic{figure}}

\setcounter{footnote}{0} 
\setcounter{section}{0}
\setcounter{table}{0}
\setcounter{figure}{0}
\end{document}